\documentclass[12pt]{amsart}
\usepackage[utf8]{inputenc} 
\usepackage[active]{srcltx}
\usepackage{a4wide}
\usepackage{amsthm,amsfonts,amsmath,mathrsfs,amssymb}
\usepackage{dsfont}
\usepackage{mathtools}
\usepackage[T1]{fontenc}
\usepackage[utf8]{inputenc}
\usepackage{enumerate}
\usepackage[left=4cm,top=4cm,right=4cm,bottom=4cm]{geometry}
\usepackage{caption}
\usepackage{subcaption}
\numberwithin{equation}{section}
\usepackage[makeroom]{cancel}

\usepackage{geometry}
\usepackage{graphicx}
\usepackage{amssymb}
\usepackage{amsmath}
\usepackage{amsthm}
\usepackage{listings}
\usepackage[colorlinks=true,urlcolor=blue,citecolor=blue,linkcolor=blue]{hyperref}
\usepackage{esint}
\usepackage{xcolor}

\usepackage{cite}
\usepackage{etex}
\usepackage[all]{xy}
\usepackage{pgf,tikz}
\usetikzlibrary{arrows}
\usepackage{pgfplots,pgfcalendar}
\usetikzlibrary{patterns}
\usepackage{hyperref}
\def\D{{\mathbb D}}  \def\T{{\mathbb T}}
\def\C{{\mathbb C}}  \def\N{{\mathbb N}}

\def\({\left(}       \def\){\right)}

\newtheorem*{theorem A}{Theorem A}
\newtheorem*{theorem B}{Theorem B}
\newtheorem*{theorem*}{Theorem}
\newtheorem{theorem}{Theorem}[section]
\newtheorem{lemma}[theorem]{Lemma}
\newtheorem{proposition}[theorem]{Proposition}

\theoremstyle{definition}
\newtheorem{definition}[theorem]{Definition}
\newtheorem{example}[theorem]{Example}

\theoremstyle{remark}
\newtheorem{remark}[theorem]{Remark}

\numberwithin{equation}{section}

\DeclareMathOperator*{\esssup}{ess\,sup}

\newcommand{\doi}[1]{\url{doi.org/#1}}
\begin{document}
%%% Title
\title[Closed Range]{Inequalities on Tent Spaces and Closed Range Integration Operators on spaces of Average Radial Integrability}
%%% Information for the first author
\author[T. Aguilar-Hern\'andez ]{Tanaus\'u Aguilar-Hern\'andez}
\address{Departamento de Matem\'atica Aplicada II and IMUS, Escuela T\'ecnica Superior de Ingenier\'ia, Universidad de Sevilla,
Camino de los Descubrimientos, s/n 41092, Sevilla, Spain}
\email{taguilar@us.es}

\author[P. Galanopoulos]{Petros Galanopoulos}
\address{Department of Mathematics, Aristotle University of Thessaloniki, 54124, Thessaloniki, Greece}
\email{petrosgala@math.auth.gr }

\subjclass[2020]{Primary 47G10, 28A; Secondary 30H, 30H10, 30H20, 30H30 46E15}

\date{\today}

\keywords{Radial integrability, tent spaces, reverse Carleson measures, closed range integration operators}

\thanks{This research was supported in part by Ministerio de Universidades: "Margarita Salas" grant (Funded by the Spanish Recovery, Transformation and Resilience Plan, and European Union - NextGenerationEU), and and Junta de Andaluc\'ia, P20-00664.}

\maketitle

\vspace{-0.2cm}
\begin{abstract}
We deal with a reverse Carleson measure inequality for the tent spaces of analytic functions in the unit disc $\mathbb{D}$ of the complex plane. The tent spaces of measurable functions were introduced by Coifman, Meyer and Stein.
 
Let $1\leq p,q < \infty$ and consider the positive Borel measure $d\mu(z) = \chi_{G}(z)\frac{dm(z)}{(1-|z|)}$  defined in terms of a measurable set $G \subseteq \D $ and of the area Lebesgue measure $dm(z)$ in $\mathbb{D}$. We prove a necessary and sufficient condition on $G$ in order to exist a constant $K>0$ such that
\begin{align*}
	\int_{\T} \left(\int_{\Gamma(\xi)} |f(z)|^{p}\ d\mu(z) \right)^{q/p}\ |d\xi|\geq K \,\int_{\T} \left(\int_{\Gamma(\xi)} |f(z)|^{p}\ \frac{dm(z)}{1-|z|}\right)^{q/p}\ |d\xi|,
\end{align*}
for any $f$ analytic in $\mathbb{D}$ with the property, the right term of the inequality above is finite. Here $\mathbb{T}$ stands for the unit circle and $\Gamma(\xi)$ is a non-tangential region with vertex at $\xi \in \mathbb{T}$\,.

 This work extends the study of D. Luecking on Bergman spaces to the analytic tent spaces. We apply our result to characterize the closed range property of the integration operator known as Pommerenke operator when acting on the average radial integrability spaces. T. Aguilar-Hern\'andez, M. Contreras and L. Rodr\'iguez-Piazza  introduced these spaces for the first time in the literature. The Hardy and the Bergman spaces form part of this family. 
\end{abstract}

\tableofcontents

\section{Introduction}
Let $\mathcal{H(\mathbb D)} $ be the class of analytic functions in the unit disc $\mathbb{D}$  of the complex plane and
  $\mathbb{X} \subset \mathcal{H(\mathbb D)}$ be a Banach space. Consider a bounded linear operator
    $$ Q: \mathbb{X} \to \mathbb{X}\,. $$ 
  We call $Q$ bounded below 
    if there is a positive constant $C$ such that 
\begin{equation}\label{defbb}
	\|Q(f)\|_{\mathbb{X}}\,\, \geq \,\,C\,\,\,\| f\|_{\mathbb{X}}\,\,\,,\qquad f \in\mathbb{X}\,.
\end{equation}

An operator that fulfills property (\ref{defbb}) has closed range and it is not compact. Conversely, if $Q$ is an one-to-one operator with closed range then it is bounded below. %Nevertheless,
%there are operators for which the closed range property is equivalent to (\ref{defbb}). 
Typical examples of operators for which the closed range property is equivalent to (\ref{defbb}) are the composition operators (see\cite{Cowen_McCluer}).
In addition, the same is true for the integration operator  
\begin{equation*}
	T_{g}(f)(z)\,=\, \int_0^z\, f(\zeta)\,g'(\zeta)\,d\zeta	
	\,\,,\,\,\qquad z\in \mathbb{D}
\end{equation*}
and its companion
\begin{equation*}
	S_{g}(f)(z)\,=\, \int_0^z\, f'(\zeta)\,g(\zeta)\,d\zeta	\,\,,\,\,\qquad z\in \mathbb{D}\,.
\end{equation*}
%where $f\in \mathbb{X}$\,. 
The function $g\in \mathcal{H}(\mathbb{D})$ is called the symbol of the operators. The operator $T_g$ is also known as Pommerenke operator.
Condition $(\ref{defbb})$, when $T_g$ is acting boundedly on a Banach space $\mathbb X$ of analytic functions in $\mathbb D$, is 
equivalent to the closed range property. 
This is also true for $S_g$ under the following modification. It is clear by its definition that $S_g$ sends the constant functions to the zero function. Therefore, it makes sense to consider $S_g$ on
$\mathbb {X}/\mathbb C$. Otherwise, it is not one-to-one. Then, we can state that 
$S_g$ is bounded below on $\mathbb {X}/\mathbb C$ if and only if it has closed range in $\mathbb {X}/\mathbb C$\, (see \cite{Anderson}\,). 
Consequently, it is natural to look for those symbols $g$ that introduce bounded below operators on Banach spaces of analytic functions.
 
 The operators $T_g, S_g$ are closely related to the pointwise multiplication operator $$M_g(f)\,=\, f \cdot\,g\,\,, \qquad f\in \mathbb{X},$$
%of 
%the space $\mathbb{X}$ 
since 
\begin{equation}\label{conection}
	M_g(f)\,=\, f(0)g(0)+ T_g(f) +S_g(f)\,\,, \qquad f\in \mathbb{X}. 
\end{equation}

Luecking, in \cite{Luecking_1981}, considered (\ref{defbb}) for the
$M_g$ in the case of the Bergman spaces. We say that an $f \in \mathcal{H}(\mathbb{D})$ belongs to the Bergman space $A^p\,,\,\, p\geq 1$\,, if 
\begin{align}
	\| f \|_{A^p}^p\,=&\, \int_{\mathbb{D}}\,\, |f(z)|^p\, dm(z) \,\nonumber\\
	\asymp & |f(0)|^p + \int_{\mathbb{D}}\, |f'(z)|^p \, (1-|z|^2)^p\, dm(z) <\infty,\,\label{Bergmandef2}
\end{align}
where $dm(z)$ stands for the normalized area Lebesque measure in the unit disc. 
Through out this work the notation $A \asymp B$ states that  
 there are positive  constants $k_1, k_2$ such that
$k_1\, A \leq B \leq k_2\, A$. If one of the inequalities is true then we use the symbol $ \lesssim $\,.
Let now a $g\in H^{\infty}$, that is a bounded analytic function in $\mathbb{D}$, since this is exactly the class of symbols that introduces a bounded multiplication operator in Bergman spaces.
Under this assumption, Luecking proved that $M_g$ is 
bounded below in $A^p$ if and only if  there is an $\eta \in (0,1)$, a $\delta >0$ and 
a $ c>0$ such that
\begin{equation}\label{main condition}
	m(G_c\cap \Delta(\alpha,\eta)) \geq\, \delta \, m( \Delta(\alpha,\eta))\,\,, \qquad \forall \alpha \in \mathbb{D}\,,
\end{equation}
where  
\begin{equation}\label{levelset1}
	G_c\,=\,\{ z \in \mathbb D : |g(z)|>c \}\,.
\end{equation}
The set
$$
\Delta(\alpha,\eta) = \left\{ z\in \mathbb{D} \, : \, \,|\phi_{\alpha}(z)|=\left|\frac{\alpha - z}{1-\bar{\alpha} z}\right|\,< \, \eta \right\}
$$
 is the pseudohyperbolic disc with "center" $\alpha$ and "radius" $\eta$\,.	
 It is clear that $z\in \Delta (\alpha, \eta) $ if and only if $\alpha \in \Delta (z,\eta)$. Moreover,  
 \begin{equation}\label{psh compar}
 (1-|\alpha|) \asymp (1-|z|)\,\,, \quad \forall z\in \Delta (\alpha, \eta)\,,
\end{equation}
 and 
  \begin{equation}\label{psh area}
  m( \Delta (\alpha, \eta)) \asymp (1-|\alpha|)^2 \,\,.
\end{equation}
 %Note that by $A \asymp B$  we mean that there are positive  constants $k, \lambda$ such that
 %it holds  $k\, A \leq B \leq \lambda\, A$.
The constants involved in (\ref{psh compar}) and (\ref{psh area}) depend on $\eta$\, (see, e.g., \cite{duren_schuster_2004} for more information about the pseudohyperbolic metric in $\mathbb D$ and its properties).
 
 However, the characterization of $M_g$ as a bounded below operator on $A^p$ is actually a consequence of the main result of \cite{Luecking_1981}  according to which,  condition (\ref{main condition}) is equivalent to
 a reverse Carleson measure type inequality for Bergman spaces. The following is presented as "Main Theorem" in \cite[p. 2]{Luecking_1981}.
 
 \begin{theorem A}
 	Let $G$ be a measurable subset of $\mathbb{D}$ and $p>0$. There is a constant $K>0$ such that 
\begin{equation}\label{main inequality}
	\int_{G}\,\,|f(z)|^p\, dm(z) \geq \,K\, \int_{\mathbb{D}}\,\,|f(z)|^p\, dm(z)\,,\quad\quad f\in A^p
\end{equation}
if and only if there is a $\eta \in (0,1)$ and a $\delta>0$ such that
 (\ref{main condition})
is true.
 \end{theorem A}
%This is presented as "Main Theorem" in p.2 of \cite{Luecking_1981}.

Employing the Bergman space norm, stated
in terms of the derivative 
%of $f$ that is 
as in (\ref{Bergmandef2}),
 and Theorem A 
% as in the case of $M_g$ 
 we  can also answer 
(\ref{defbb}) for a bounded $S_g$ in $\mathbb{X} = A^p/\mathbb{C}$
and for a bounded $T_g$ in $\mathbb{X} = A^p$. We answer the problem in terms of condition (\ref{main condition}) on the sets (\ref{levelset1})
%$$G_r\,=\,\{ z \in \mathbb D : |g(z)|>r \}$$
and on the sets
\begin{equation*}
G_c\,=\,\{ z \in \mathbb D : |g'(z)| \,(1-|z|)\,>c \}
\end{equation*}
respectively. We recall that $S_g$ is bounded on $A^p$ if and only if $g\in H^{\infty}$.
On the other hand, $T_g$ is bounded if and only if the symbol $g$ belongs to the Bloch space
that is 
\begin{equation*}
	\|f\|_{\mathcal{B}}= |g(0)|+ \sup_{z\in \mathbb{D}}\,|g'(z)|\,(1-|z|^2) < \infty\,
\end{equation*} 
(see \cite{AS2}).
%It is natural  to ask question (\ref{defbb})
%for other well known spaces $\mathbb{X}$. 

Anderson, in \cite{Anderson},
proved that in the case of the Bloch, the BMOA  and the Hardy spaces
$H^p$  
 only the constant symbols introduce a bounded below operator $T_g$. The BMOA consists of those $f\in \mathcal{H}(\mathbb{D})$ that have boundary values almost every where in the unit circle and those boundary values define a square-integrable function on the unit circle with bounded mean oscillation. We will not define Hardy spaces at this point. In the theory of spaces of analytic functions, %in the unit disc 
 the Bloch and the BMOA space stand for the natural limit case, as $p$ grows to infinity, of the Bergman and  Hardy spaces respectively.
More information about the Bloch and the BMOA space can be found in \cite{Zhu}, \cite{Gi}.

Based on the above, someone may come down to the conclusion that (\ref{defbb}) for $S_g$,  when acting on the Bloch, BMOA and Hardy spaces,  gets trivial answer too.
On the contrary,  Anderson in \cite{Anderson} proves that this not the case when $\mathbb{X}$ is
the Hilbert Hardy space $H^2$.  Due to the fact that $H^2$ can be described equivalently in terms of the  Littlewood-Paley identity as
\begin{align*}
	\|f\|^2_{H^2}\,= & \sup_{r\in (0,1)}\,\int_{0}^{2\pi} |f(re^{i\theta})|^2\,d\theta\\
	\asymp & \,|f(0)|^2\,+\,\int_{\mathbb{D}}\, |f'(z)|^2\,(1-|z|)\,dm(z) <\infty\,,
\end{align*}
the author is able to apply the technique used in the case of Bergman spaces  and completely answer the problem as follows. Let a $g\in H^{\infty}$, that is, consider  
$S_g$ to be bounded in $H^2/\mathbb{C}$. Anderson proves that condition (\ref{main condition}) on the sets (\ref{levelset1})
%$G_r\,=\,\{ z \in \mathbb D : |g(z)|>r \}$
is necessary and sufficient in order for $S_g$ to be bounded below in $H^2/\mathbb{C}$.
 
 Looking carefully at the result of the $H^2$ space and that of Bergman spaces, we can immediately state that $S_g$ is bounded below in $H^2/\mathbb{C}$ if and only if  $S_g$ and $M_g$ are bounded below in $A^{p}/\mathbb{C}$ and $A^{p}$, respectively. However, it is very interesting the following fact. Saying that $S_g$ is bounded below in $H^2/\mathbb{C}$ is not equivalent to the boundedness from below of $M_g$ in $H^2$. The latter is true if and only is the radial limit function of $g\in H^\infty$ is essentially bounded away from $0$ on $\T$, which is a weaker condition (see  \cite[p. 93]{Anderson}).
 
Anderson goes further. He precisely characterizes the bounded functions under discussion. The functions $g\in H^{\infty}$, for which (\ref{defbb}) is true for $S_g$ in $H^2/\mathbb{C}$, are exactly those that can be factored as 
$$
g=B\cdot F\,,
$$  
where $B$ is a finite product of interpolating Blaschke products
and $F,\,1/F\in H^{\infty}$\,.  In its turn this is equivalent to the existence of an $r>0$ and a $\eta \in (0,1)$ such that 
$$
\sup_{z\in \Delta(\alpha,\eta)}\,|g(z)|\,> r\,\,,\qquad \forall \alpha \in \mathbb{D}\,.
$$
Finally, he establishes that the boundedness from below of $S_g$ in the $ \mathcal{B}/\mathbb{C}$
is characterized by the same $g$ as those in the Hardy case (see \cite[Theorem 3.9]{Anderson}). 

Recently, Panteris  extended  the result of Anderson about $S_g$  on $H^2/\mathbb{C}$ to any Hardy space $H^p/\mathbb C$, $p>1\,$ \cite{Panteris}. His approach is based on the use of the following equivalent characterization of $H^p$ due to  Calderon (see \cite[Theorem 1.3]{Pavlovic}):
 \begin{align*}
	\|f\|^p_{H^p}\,= & \sup_{r\in (0,1)}\,\int_{0}^{2\pi} |f(re^{i\theta})|^p\,d\theta\\
	\asymp & \,|f(0)|^p \,+\,\int_{\mathbb{T}}\,\(\int_{\Gamma(\xi)}\, |f'(z)|^2\,\,dm(z)\)^{p/2}\,|d\xi| \,.
\end{align*}
The symbol $\mathbb{T}$ stands for the unit circle and $\Gamma(\xi)$  
is an angular region 
%$\Gamma_{M}(\xi)=\{z\in \D\ :\ |1-\overline{\xi}z|<M(1-|z|^2)\}$, where $M>1/2$ and %$\xi\in\T$, is an analogue of the Stoltz angle
 with vertex at $\xi \in \mathbb{T}$.
 Using appropriately this tent space expression, Panteris arrived to the same conclusion as Anderson. He also deals with the case $\mathbb{X}$= BMOA and that of the Besov spaces.

Our aim is to extend the study of property
(\ref{defbb}) for $T_g$ and $S_g$ to the more general setting of the $RM(p,q)$ spaces,\, $1\leq p,q<\infty$. 
The first author, M. Contreras and L. Rodríguez-Piazza, motivated by the property of bounded radial integrability 
\begin{equation}\label{BAI}
\sup_{\theta\in [0,2\pi]}\, \int_{0}^{1}\,|f(re^{i\theta})|\,dr \,<\infty\,,
\end{equation}
introduced in \cite{Aguilar-Contreras-Piazza} the spaces of average radial integrability $RM(p,q)$. 
%According to Fejer-Riesz inequality condition (\ref{BAI}) is satisfied by all $f\in H^1$, see \cite{Du}.

Let $1\leq p,q <\infty$, 
we say that an $f\in \mathcal{H}(\mathbb{D})$ belongs to $RM(p,q)$ if  
\begin{equation*}
	\begin{split}
	\|f\|^q_{RM(p,q)} = &\frac{1}{2\pi}\int_{0}^{2\pi} \left(\int_{0}^{1} |f(r e^{i \theta})|^p \ dr \right)^{q/p}d\theta  \,<\infty\,.\\
		%\rho_ {p,\infty}(f)&=\esssup_{t\in[0,2\pi)}\left(\int_{0}^{1} |f(r e^{i t})|^p \ dr %\right)^{1/p}, \quad \text{ if } p<+\infty, \\
		%\rho_ {\infty,q}(f)&=\left(\frac{1}{2\pi}\int_{0}^{2\pi} \left(\sup_{r\in [0,1)} |f(r %e^{i t})| \right)^{q}dt\right)^{1/q},\quad\text{ if } q<+\infty,\\
	%	\rho_{\infty,\infty}(f)&=\|f\|_{H^{\infty}}=\sup_{z\in\D}|f(z)|.
	\end{split}
\end{equation*}
%They denoted these spaces as $RM(p,q)$.
Among them  we can identify the Bergman spaces when $p=q$. This is justified only by a change in the order of integration. 
When $p\neq q $, an application of Holder's inequality leads to the conclusion that each $RM(p,q)$ space is located in between two Bergman spaces. 
Furthermore, when $p=\infty$, the $RM(\infty,q)$ spaces are the Hardy spaces $H^q$. Condition (\ref{BAI}), in terms of the notation used for the spaces of average radial integrability, can be visualized as the limit case $RM(1,\infty)$ and, according to Fejer-Riesz inequality, (\ref{BAI})  is satisfied by all $f\in H^1$ (see \cite{Du}). In \cite{Aguilar-Contreras-Piazza} it is proved that $H^s \subseteq RM(p,q)$, $1\leq p,q<\infty$ 
if and only if $1/s \leq 1/p +1/q\,.$
%This gives the impression that they 
%can be well understood easily. 
%Nevertheless, they have already attracted the attention of the researchers. 
%The $RM(p,q)$  appear in many problems of the area in a natural way. One reason is that 
On the other hand, the $RM(p,q)$ spaces can be identified to 
the analytic tent spaces 
$$AT^q_p=\mathcal{H(\mathbb{D})}\cap T^q_p\,\,,\, $$ for $1\leq p,q<\infty$ (see \cite{Aguilar-Galanopoulos}). The tent spaces $T^q_p$, $1\leq p,q<\infty$, consist of those measurable functions $f$ in $\mathbb{D}$ that 
$$
\|f\|^q_{T^q_p} = \int_{\mathbb{T}}\,
\(\int_{\Gamma(\xi)}\, |f(z)|^p\,\,\frac{dm(z)}{1-|z|}\)^{q/p}\,|d\xi| <\infty\,.
$$
%In addition, if we pose classical questions of the area in the $RM(p,q)$  setting it seems that they can not be confronted by the usual techniques.
We will 
provide all the information needed about $RM(p,q)$ and $AT_p^q$ in the next section.

Taking into account all the above, it is natural to consider problem (\ref{defbb}) for the $S_g$, $T_g$ and $M_g$ operators when acting on the $RM(p,q)$ spaces. The key, in order to  answer this question, is the following
%to work with their equivalent expression as $AT^q_p$ spaces.
%that the $RM(p,q)$ have a strong connection to the Bergman spaces it is natural to think that (\ref{defbb}) for $S_g, T_g, M_g$ should receive the same answer as in the Bergman case. 
%Eventually, this is the case. However, bringing in mind (\ref{main condition}) we realize that it is an area based answer. On the other hand, due to the formula of definition of the $RM(p,q)$, it is nor clear to us how to reach to (\ref{main condition}) using  (\ref{RMfirst}) directly. 
%We overcome that obstacle by employing the equivalent tent expression of the $RM(p,q)$. First, 
analogue of (\ref{main inequality}) for the $AT^q_p$\, spaces. This is our main result.
\begin{theorem}\label{Main Theorem}
	Let $1\leq p,q<+\infty$ and $G\subset \D$ be a measurable set. The following assertions are equivalent:
	\begin{enumerate}
		\item[(a)] There is a constant $K>0$ such that
		\begin{align*}
			\left(\int_{\T} \left(\int_{\Gamma(\xi)\cap G} |f(z)|^{p}\ \frac{dm(z)}{1-|z|}\right)^{q/p}\ |d\xi|\right)^{1/q}\geq K \|f\|_{T_{p}^{q}}\,,\quad\quad f\in AT_{p}^{q}\,.
		\end{align*}
	%	for all $f\in AT_{p}^{q}$\,.
		\item[(b)] There exist an $\eta \in (0,1)$ and a $\delta>0$ such that
		\begin{align*}
			m(G\cap \Delta(\alpha,\eta))\,\geq \,\,\delta \,\, m(\Delta(\alpha,\eta))
		\end{align*}
		for all $\alpha \in \D$.
	%	\item[(c)] There exist $\delta>0$ and $\eta\in (0,1)$ such that
	%	\begin{align*}
	%		m(G\cap \Delta_\eta(\alpha))\geq \delta m(\Delta_\eta(\alpha))
	%	\end{align*}
	%	for $\alpha \in \D$.
	\end{enumerate}
\end{theorem}
%This is actually a reverse Carleson inequality problem for the measure $\chi_G(z)\, \frac{dm(z)}{1-|z|}$ in  the tent spaces setting and we provide an answer.
%to extend the study of property
%(\ref{defbb}) for $T_g$, $S_g$ to the more general setting of $RM(p,q)$,\, $p,q\in [1,\infty)$.
%The next step,  in order to characterize the closed range $T_g, M_g, S_g$ operators, is to employ the equivalence of $RM(p,q)$ with $AT^q_p$ and apply Theorem 1. The answer is stated below. 
Now, we are in position to characterize the closed range operators $T_g$, $S_g$ and $M_g$ as follows.
\begin{theorem} \label{CRTg}
	Let $1\leq p,q<\infty$ and $g\in \mathcal{B}$. The operator $T_g: RM(p,q) \rightarrow RM(p,q)$ has closed range if and only if there is an $\eta \in (0,1)$, a $\delta>0$  and a $c>0$ such that
	\begin{align}\label{conditionTg1}
		m(G_c\cap \Delta(\alpha,\eta))\geq \delta m(\Delta(\alpha,\eta)),\quad \alpha\in\D,
	\end{align}
	where $G_c=\{z\in\D\ :\ |g'(z)|(1-|z|)>c\}$.
	%, for every $\alpha\in\D$\,\,.
\end{theorem}

\begin{theorem} \label{CRSgMg}
	Let $1\leq p,q<\infty$ and $g\in H^\infty$. The following are equivalent:
	\begin{enumerate}
		\item 
		$M_g: RM(p,q)\rightarrow RM(p,q)$ has closed range.
		\item $S_g : RM(p,q)/\mathbb C \rightarrow RM(p,q)/\mathbb C$
		has closed range.
		\item There is an $\eta \in (0,1)$, a $\delta>0$  and a $c>0$ such that
		\begin{align*}
			m(G_c\cap \Delta(\alpha,\eta))\geq \delta\, m(\Delta(\alpha,\eta)),\quad \alpha\in\D,
		\end{align*}
		where $G_c=\{z\in\D\ :\ |g(z)|>c\}$.
		
	\end{enumerate}

\end{theorem}
Notice that, the symbols $g\in H^{\infty}$ of Theorem~\ref{CRSgMg} are explicitly understood  by the work of Anderson. Nevertheless, as far as we know,
there is no example of Bloch functions that fulfil condition (\ref{conditionTg1}). Notice that constant functions are not among them.
Recall that, according to Anderson, the function $g\in BMOA$ that introduce a closed range $T_g$ on the Hardy spaces are exactly the constant functions. That being the case, we  provide a non trivial collection of $g\in \mathcal B$ that serve as symbols of a closed range $T_g$ in the $RM(p,q)$ spaces. Consequently, this collection of symbols is also valid for the case of Bergman spaces.

This work has the following structure. The first part of Section $2$ is about the definition of $RM(p,q)$, $AT_p^q$ spaces and their connection. We also present equivalent expressions in terms of the derivative. These expressions are known in the literature as Littlewood-Paley type inequalities. In the rest of the section we discuss the background 
related to the action of $T_g$, $S_g$ and $M_g$ on the $RM(p,q)$. 
In Section $3$, we introduce the auxiliary lemmas needed and we present the proofs of Theorem~\ref{Main Theorem}, Theorem ~\ref{CRTg} and Theorem~\ref{CRSgMg}. 
The last section is devoted to the presentation of examples of symbols $g$ related to Theorem~\ref{CRTg}. To be specific, we prove that a $g\in \mathcal B$ defined in terms of lacunary series introduces a closed range operator $T_g$ in the $RM(p,q)$ spaces. However, this is not true for univalent functions $g\in \mathcal B$.

\section{Definitions and first properties}
\addtocontents{toc}{\protect\setcounter{tocdepth}{1}}
In this section we recall the definition of the spaces of analytic functions under discussion. Moreover, we compile some properties for the sake of being self-contained.
\subsection{Average radial integrability spaces}
These are the $RM(p,q)$ spaces introduced in \cite{Aguilar-Contreras-Piazza}. Although we are interested  for their Banach space version, below we present their definition for the full range of the values of $p,q$.
\begin{definition}
	Let $0<p,q\leq +\infty$. We define the spaces of analytic functions
	
\begin{equation*}
RM(p,q)=\{f\in \mathcal{H}(\D)\ :\ \rho_{p,q}(f)<+\infty\}
\end{equation*}
	where
	\begin{equation*}
	\begin{split}
	\rho_ {p,q}(f)&=\left(\frac{1}{2\pi}\int_{0}^{2\pi} \left(\int_{0}^{1} |f(r e^{i \theta})|^p \ dr \right)^{q/p}d\theta \right)^{1/q}, \quad \text{ if } p,q<+\infty,\\
	\\
	\rho_ {p,\infty}(f)&=\esssup_{t\in[0,2\pi)}\left(\int_{0}^{1} |f(r e^{i\theta})|^p \ dr \right)^{1/p}, \quad \text{ if } p<+\infty, \\
	\\
	\rho_ {\infty,q}(f)&=\left(\frac{1}{2\pi}\int_{0}^{2\pi} \left(\sup_{r\in [0,1)} |f(r e^{i \theta})| \right)^{q}d\theta \right)^{1/q},\quad\text{ if } q<+\infty,\\
	\\
	\rho_{\infty,\infty}(f)&=\|f\|_{H^{\infty}}=\sup_{z\in\D}|f(z)|.
	\end{split}
	\end{equation*}
\end{definition}

Our interest is restricted to the range $p,q \in [1,\infty)$. These $RM(p,q)$ are Banach spaces under the norm 
$$
\|f\|_{RM(p,q)} = \rho_ {p,q}(f)\,.
$$
Notice that, when $p=q$ 
$$
RM(p,p) \equiv A^p\,.
$$
A change in the order of integration is enough to prove it. In any other case, an appropriate application of Holder's inequality results to the containments
$$
A^p \subset RM(p,q)\, \subset A ^q,\quad q<p
$$
or 
$$
A^q\subset RM(p,q) \subset A^p \,,\quad p<q\,.
$$
The inclusions are strict.

The case $q=\infty$ stand for the bounded $p$-integrability of $f$. It makes no difference if instead of the essential supremum we use the supremum.

It is worth of mentioning that in the case $p=\infty$ we have
$$
RM(\infty,q)\equiv H^q\,
$$
due to the fact that the membership of an $f$ in $H^q$ can be equivalently described by the $q$-integrability over $[0,2\pi)$ of the radial maximal function 
$$
Rf(e^{i\theta})=\sup_{r\in [0,1)} |f(r e^{i \theta})|\,
$$
of the function $f$ (see \cite{Du}). 

Closing, we recall from \cite{Aguilar-Contreras-Piazza} the pointwise growth estimates
\begin{equation*}
	|f(z)| \leq \,C_1\,\frac{\|f\|_{RM(p,q)}}{(1-|z|)^{1/p+1/q}} \,\,,\quad z\in \D
\end{equation*}
and
\begin{equation}\label{ged} 
	|f'(z)| \leq \,C_2 \,\frac{\|f\|_{RM(p,q)}}{(1-|z|)^{1/p+1/q+1}} \,\,,\quad z\in \D
\end{equation}
of the $RM(p,q),\,1\leq p,q <\infty\,$. The $C_1, C_2$ are positive constants
independent of $f$ and $z$\,.

\subsection{Tent spaces }
The work of Coifman, Meyer and Stein  is considered as the starting point of the study of tent spaces \cite{CMS}. 
Since then, they have been widely studied by many authors (see, e.g., \cite{Luecking1987}, \cite{Luecking},   \cite{Jevtic_1996}, \cite{Arsenovic}, \cite{Cohn_Verbitsky_200}, \cite{Perala_2018}, \cite{MiihkinenPauPeralaWang2020}).

Let a $\xi\in \T$. We define the cone like region
\begin{align*}
\Gamma_{1/2}(\xi) =\bigl\{ z\in \mathbb{D} : |z|< 1/2   \bigr\} \cup \bigcup_{|z|<1/2}[z,\zeta)\,. 
\end{align*}
\begin{definition}
	Let $0<p,q, <+\infty$. The tent spaces  $T_p^q$ consist of measurable functions $f$ on $\D$ such that 
	\begin{align}\label{TentA}
	\|f\|_{T_{p}^{q}}= \left\{\int_{\T} \left(\int_{\Gamma_{1/2}(\xi)} |f(z)|^p \ \frac{dA(z)}{(1-|z|^2)^{}} \right)^{q/p}\ |d\xi|\right \}^{1/q}<+\infty.
	\end{align}
 \end{definition}

It is of great importance that in \eqref{TentA} we can  use any of the cone-like regions
\begin{align*}
	\Gamma_{\beta}(\xi) =\left\{ z\in \mathbb{D} : |z|< \beta   \right\} \cup \bigcup_{|z|<\beta}[z,\zeta)\,,\quad \beta\in(0,1).
\end{align*}
In other words, we have 
\begin{align*}
	\|f\|_{T^q_p} \asymp \left\{\int_{\T} \left(\int_{\Gamma_{\beta}(\xi)} |f(z)|^p \ \frac{dA(z)}{(1-|z|^2)^{}} \right)^{q/p}\ |d\xi|\right \}^{1/q} \,.
\end{align*}
Actually, we can use any of the non-tangential regions 
 \begin{align*}
		 \Gamma_M (\xi) =\left\{z\in\D: |z-\xi|< M (1-|z|^2)  \right\},\quad  
          M>\frac{1}{2}\,\,
\end{align*}
as well.
This is true due the following technical lemma, well known to the experts of the area. 
The symbol $\Gamma_C(\xi)$ stands for any of the $\Gamma_{\beta} (\xi)$ or $\Gamma_M (\xi)$\,.
\begin{lemma}{\cite[Lemma 4, p. 66]{Arsenovic}}\label{estimate 1} 
	Let $0<p,q<+\infty$, $\lambda>\max\{1,p/q\}$ and $\mu$ be a positive Borel measure on $\D$\,. There are constants $K_i=K_i(p,q,\lambda,C),\,i=1,2 $ such that
	\begin{align*}
K_1 \int_{\T} \mu (\Gamma_C(\xi))^{q/p} |d\xi|\leq\int_{\T}\left(\int_{\D} \left(\frac{1-|z|}{|1-z\overline{\xi}|}\right)^{\lambda} \ d\mu(z)\right)^{q/p}|d\xi|\leq K_2 \int_{\T} \mu (\Gamma_C(\xi))^{q/p} |d\xi|\,.
	\end{align*}
	%, there is a constants $C_{1}=C_{1}(p,q,\lambda,C)$ and $C_{2}=C_{2}(p,q,\lambda,C)$ %such that
	%\begin{align*}
%	C_1\int_{\T} \mu(\Gamma (\xi))^{q/p} |d\xi|\leq\int_{\T}\left(\int_{\D}
  % \left(\frac{1-|z|}{|1-z\overline{\xi}|}\right)^{\lambda} \
  % d\mu(z)\right)^{q/p}|d\xi|\leq C_{2} \int_{\T} \mu (\Gamma(\xi))^{q/p} |d\xi|,
%	\end{align*}
%	for every positive measure $\mu$ on $\D$. 
\end{lemma}

 	 As in the case of the $RM(p,q)$,  the tent spaces are defined for the limit values of the $p,q$ as well. The $T_{\infty}^{q}$ consists of measurable functions $f$ on $\D$ with
	 	 \begin{align*}
	 	 	\|f\|_{T_{\infty}^{q}}=\left\{\int_{\T} (\esssup_{z\in \Gamma_{1/2}(\xi)} |f(z)|)^{q}\  |d\xi|\right\}^{1/q},\quad\text{ if } q<+\infty\,.
	 	 \end{align*}
 	  It is known that the definition is independent of the type of the non-tangential region we use.
 	   % in the definition.
	 	 %	where the essential supremum is taken with respect to the measure $(1-|z|^2)^{\alpha-2}\ %dA(z)$. Notice that the definition of $T_{\infty}^{q}(\alpha)=T_{\infty}^{q}$ is independent % of $\alpha$.
	 
	 When $q=+\infty$ and $p<+\infty$, the tent space $T_{p}^{\infty}$ consists of measurable functions $f$ on $\D$ with
	 	 \begin{align}\label{TentC}
	 	 	\|f\|_{T_p^\infty}=\sup_{\xi\in \T} \left(\sup_{w\in\Gamma(\xi)} \frac{1}{(1-|w|^2)} \int_{S(w)} |f(z)|^{p}\, dA(z) \right)^{1/p}<+\infty,
	 	 \end{align}
	 	 where 
	 	 $$S(re^{i\theta})=\left\{\rho e^{it}\ : 1-\rho\leq 1-r,\ |t-\theta|\leq \frac{1-r}{2}\right\}$$ 
	 	 for $re^{i\theta}\in \D\setminus\{0\}$ and $S(0)=\D$.
	 	 An equivalent way to state (\ref{TentC}) is to say that the measure
	 	 $$d\mu(z)=|f(z)|^p\,dm(z)$$
	 	 is a Carleson measure. Recall that a positive Borel measure $\mu$ in $\mathbb{D}$
	 	 is called Carleson if and only if 
	 	 $$\sup_{I\subset \mathbb{T}} \frac{\mu(S(I))}{|I|}\,<\infty \,\, .$$
	 	  The supremum is taken over all arcs $I$ of the unit circle, $|I|$ is the arc length and $$S(I)=\left\{ z\in \mathbb{D} :\,\, 1-|I|\leq \,\,|z| <1\,,\,\, \frac{z}{|z|} \in I \right\}\,.$$ 
	 	 %	\begin{align*}\label{CM}
	 	 %	\sup_{I\subset \mathbb{T}} \frac{\mu(S(I))}{|I|}\,<\infty \,\, .
	 	 %	\end{align*}
 	%\begin{theorem}
%	Let $1\leq q<+\infty$. Then $H^{q}=RM(\infty,q)=AT_{\infty}^{q}.$
%\end{theorem}
%\begin{proof}
%	Fix $M>1/2$ and $f\in\mathcal{H}(\D)$. By \cite[Theorem 17.11(a), p. 340]{rudin_real_1987}, we have that there is a constant $\gamma=\gamma(M)$ such that 
%	\begin{align*}
%	\rho_{\infty,q}^{q}(f)&=\int_{\T} \left(\sup_{r\in [0,1)} |f(r\xi)|^{q}\right)\ |d\xi|\leq \int_{\T} \left(\sup_{z\in \Gamma_{M}(\xi)} |f(z)|^{q}\right)\ |d\xi|=\|f\|_{T_{\infty}^{q}}^{q}\leq \ \gamma \|f\|_{H^q}^{q}.
%	\end{align*}
%	Since $RM(\infty,q)=H^q$ (see \cite{Aguilar-Contreras-Piazza}), we are done.
%\end{proof}

	%\begin{remark}\label{independentTinftyq}
%	One can obtain an equivalent norm in $AT_{\infty}^{q}$ replacing in \eqref{TentB} the set %$\Gamma(\xi)$ by any other Stolz region $\Gamma_{M}(\xi)$.
%\end{remark}

	Here, we are interested in the holomorphic version of the tent spaces. We denote that  as 
	$$AT_p^{q} = T_p^q\cap \mathcal{H}(\D).$$
	 The limit case $p=\infty$ corresponds to the Hardy spaces, that is, $AT_{\infty}^q\equiv H^q \,,$ 
	 since the membership of an $f\in H^q$ can be determined by the $q$-integrability over $[0,2\pi)$ of the non-tangential maximal function
	 $$
	 N_{C}f(e^{i\theta})=\esssup_{z\in \Gamma_{C}(e^{i\theta})} |f(z)|\,
	 $$
	of the function $f$ (see \cite{Du}).
%the non-tangential regions 
%\begin{align*}
%	\Gamma (\xi) & = \Gamma_M (\xi) =\left\{z\in\D: |z-\xi|< M (1-|z|^2)  \right\},
%\end{align*}
%where $ M>\frac{1}{2}$ and 

\subsection{Integration Operators on the RM(p,q) spaces}
Let $g\in \mathcal H (\mathbb D)$. We recall from the introduction  the definition of  the integration operator $T_g$ as 
$$
T_g(f)(z)= \int_0^z f(\zeta)\,g'(\zeta)\,d\zeta\,\,,\quad\quad z\in\mathbb D\,.
$$
We consider its action on a Banach space $\mathbb X\subset \mathcal H (\mathbb D) $, that is, $f\in \mathbb X$.

Pommerenke was the first who considered $T_g$ on spaces of analytic functions by proving that it is bounded on $H^2$
if and only if $g\in BMOA$ \cite{Pom}. Aleman and Siskakis proved in \cite{AS} that the same is true for any $H^p\,,\,p\in[1,\infty)$. Later on, they confronted the question of boundedness on the Bergman spaces $A^p,\,p\in [1,\infty)$, by establishing as  necessary and sufficient condition the belonging of the symbol $g$ to the Bloch space $\mathcal B$ in \cite{AS2}.
Hardly speaking, the key for the proof of sufficiency in any of these spaces is the use of a norm expression based on the derivative of $f$. Necessity comes out using an appropriate family of test functions. More or less this is the way to deal with the study of $T_g$ and $S_g$
in spaces of analytic functions.
 
Following this research program in \cite{Aguilar-Contreras-Piazza_2} the authors, among other things, consider the problem of the boundedness of $T_g$ in the $RM(p,q)$ spaces. For this purpose, they establish that the norm associated to the spaces can be equivalently described by the derivative. This is what they call Littlewood-Paley type inequalities. Precisely, 
\begin{equation}\label{LP1}
	\rho_{p,q}(f) \simeq \,|f(0)|\, +\,\rho_{p,q}(f'(z) (1-|z|))\,\,,\quad 1\leq p,q <\infty\,.
\end{equation}
%\begin{align}\label{LP1}
%& \left(\frac{1}{2\pi}\int_{0}^{2\pi} \left(\int_{0}^{1} |f'(re^{i\theta})|^p (1- %r)^p\ dr \right)^{q/p}d\theta \right)^{1/q}  \nonumber\\
%\nonumber\\	
% & \lesssim \quad	\left(\frac{1}{2\pi}\int_{0}^{2\pi}\left(\int_{0}^{1} |f(r e^{i %\theta})|^p \ dr \right)^{q/p}d\theta \right)^{1/q} \nonumber\\
%& \quad \quad \lesssim |f(0)| + \left(\frac{1}{2\pi}\int_{0}^{2\pi} \left(\int_{0}^{1} %|f'(re^{i\theta})|^p (1- r)^p\ dr \right)^{q/p}d\theta \right)^{1/q}\,\,,\quad 1\leq %p,q <\infty\,.
%\end{align}
That being the case, 
they follow the path just described above to prove that 
%they are able to deal with the part of the sufficient condition.
% On the other hand, for the necessity, they introduce a family of test functions for the $RM(p,q)$\,.
% As a consequence, they manage to prove that 
  $$T_g : RM(p,q) \rightarrow RM(p,q) \,,\quad \quad p,q \in [1,\infty)\,\,$$ 
boundedly if and only if $g\in \mathcal B$. 

It is well known among the experts of the area that the  boundedness  of 
$$
S_g(f)(z) =  \int_0^z f(\zeta)\,g'(\zeta)\,d\zeta\,\,,\quad\quad z\in\mathbb D\,,
$$
on Hardy and Bergman spaces
%, following the same technique as for $T_g$, 
is characterized by the condition $g\in H^{\infty}$.
This is the case for the $RM(p,q), 1\leq p,q < \infty,$ too. However, this operator is not considered in \cite{Aguilar-Contreras-Piazza_2}\,. Therefore we present the proof. 
\begin{proposition}\label{Sg}
	Let $1\leq p,q<\infty$ and $g\in \mathcal H(\mathbb D)$. The operator $S_g$ is bounded in the 
	$RM(p,q)$ spaces if and only if $g\in H^{\infty}$\,.
	\end{proposition}

\begin{proof}
	We begin with the assumption that $g\in H^{\infty}$ and we recall the (\ref{LP1}). Then
\begin{align*}
\|S_g (f)\|_{RM(p,q)} & =\rho_{p,q} (S_g(f)) \asymp |S_g(f)(0)| + \rho_{p,q} (S_g(f)'(z) (1-|z|))\\
& = \rho_{p,q} (S_g(f)'(z) (1-|z|))\\
	& = \left(\frac{1}{2\pi}\int_{0}^{2\pi} \left(\int_{0}^{1} |f'(re^{i\theta})|^p\, |g(re^{i\theta})|^p \, (1- r)^p\ dr \right)^{q/p}\,d\theta \right)^{1/q}\\
	& \leq \|g\|_{H^{\infty}} \,\,\left(\frac{1}{2\pi}\int_{0}^{2\pi} \left(\int_{0}^{1} |f'(re^{i\theta})|^p\, \, (1- r)^p\ dr \right)^{q/p}\,d\theta \right)^{1/q}\,\\
	& \lesssim \|g\|_{H^{\infty}}\,\, \|f\|_{RM(p,q)}\,,
\end{align*}
for every $f\in RM(p,q)$.
So, the sufficiency is proved. 

For the necessity we consider the operator bounded. So, there exists
a constant $C>0$ such that 
$$ \|S_g(f)\|_{RM(p,q)} \leq C \|f\|_{RM(p,q)}\,\,, \quad f\in RM(p,q)\,. $$
Combining that with (\ref{ged}) we get 
$$
|S_g(f)'(z)|\, \leq \, C'\,\frac{\|f\|_{RM(p,q)}}{(1-|z|)^{1/p+1/q+1}}\,,\quad f\in RM(p,q)\,,
$$
where $C'$ is a positive constant.
Now, we employ the test functions $$f_{\alpha}(z)=\frac{(1-|\alpha|)^{\gamma-1/p-1/q}}{(1-\bar{\alpha}z)^{\gamma}}\,,\quad \gamma > 1/p+1/q\,, \,\,\alpha \in \mathbb D\,,\, z\in \mathbb D\,.$$ 
They have the property 
$\sup_{\alpha} \|f_{\alpha}\|_{RM(p,q)} < \infty \,$ (see \cite{Aguilar})\,.
As a consequence, for any $z \in \mathbb D$,
$$
|f_{\alpha}'(z)| \,|g(z)|\, \lesssim \frac{1}{(1-|z|)^{1/p+1/q+1}}
$$
for every $ \alpha \in \mathbb D$\,. Choosing $z=\alpha$ implies that $g\in H^{\infty}$\,.
\end{proof}

Remember that $M_g$, $T_g$ and $S_g$ are connected as in (\ref{conection})\, and that $H^{\infty}\subset \mathcal B$. Therefore, without to much effort, we can check that $M_g$ is bounded in the
$RM(p,q)$ if and only if $g\in H^{\infty}$\,.

As it is stated in the introduction, our concern is the characterization of $T_g$, $S_g$ and $M_g$ as bounded below  operators when acting boundedly on the $RM(p,q)$. 
We are able to accomplish this task 
by studying the  problem in the tent space version of the $RM(p,q)$. We claim that this  is possible since, as it is established in \cite{Aguilar-Galanopoulos}, for  $1\leq p,q < \infty$
\begin{align*}
	\rho_{p,q} (f) &\asymp \|f\|_{T_p^q}\,\,,
\end{align*}
\begin{align*}	
	 \rho_{p,q}(f'(z) (1-|z|)) & \asymp \|f'(z) (1-|z|)\|_{T_p^q}
\end{align*}
and due to the Littlewood-Paley type inequalities 
\begin{equation*}
	\|f \|_{T^q_p} \simeq \|f'(z) (1-|z|)\|_{T^q_p }\,\,,\quad 1\leq p,q <\infty\,
\end{equation*}
for every $f\in AT^q_p$\, (see \cite[Theorem 2, p. 9]{Perala_2018}).

\section{Reverse Carleson Type Measures for the $AT^q_p$\,, $1\leq p,q<\infty$}

In this section we prove our main result, Theorem~\ref{Main Theorem}, and its consequences. In the statement presented in the introduction we only make use of pseudohyperbolic discs. However, due to the technicalities we have to deal with, we need the euclidean discs 
\begin{align*}
	\Delta_{\eta}(\alpha)=\{\,z\in \mathbb D : |z - \alpha | < \eta (1-|\alpha|)\,\}\,,\quad \eta \in (0,1)\,,\,\alpha \in \mathbb D\,,
\end{align*} 
as well. Notice the 
% connection of $\Delta_{\eta}(\alpha)$ with a pseudohyperbolic disc presented in the   
following remark.
\begin{remark}\label{RemarkDisks}
		It is known (see \cite[p. 4]{Luecking_1981}) that if $\eta \in (0,1)$, $\alpha\in \D$ and $\frac{2\eta}{1+\eta^2}\leq r<1$ then
		\begin{align*}
			\Delta_{\eta}(\alpha)\subseteq \Delta(\alpha,r).
		\end{align*}
	In addition to that, if $\eta \in (0,1)$ then there is a $\beta \in (1/2,1)$
	such that $$\cup_{\alpha \in \Gamma_{1/2}(\xi)} \Delta_{\eta}(\alpha) \subseteq \Gamma_{\beta} (\xi)$$
	where $\xi\in \mathbb T$ is the same for both regions. This is the region $\Gamma_{\beta}(\xi)$ considered
	in the following two Lemmas\,.
	
	\end{remark}

%Before proceeding to the proof of Theorem~\ref{Main Theorem} we provide the following %three Lemmas which serve as building blocks.
  
 \begin{lemma}\label{L1}
Let $1\leq p<\infty $, $\eta \in (0,1)$, $\varepsilon>0$, $f\in \mathcal{H}(\D)$, and
\begin{align*}
\mathcal{A}= \left\{\alpha\in\D\ : |f(\alpha)|^{p}<\frac{\varepsilon}{m(\Delta_\eta(\alpha))} \int_{\Delta_\eta(\alpha)} |f(z)|^{p}\ dm(z)\right\}.
\end{align*}
Then, there is a constant $C_1=C_1(\eta)>0$ such that
\begin{align*}
\int_{\mathcal{A}\cap \Gamma_{1/2}(\xi)} |f(z)|^{p}\ \frac{dm(z)}{1-|z|}\leq \, \varepsilon\, C_1 \int_{ \Gamma_{\beta}(\xi)} |f(z)|^{p}\ \frac{dm(z)}{1-|z|},
\end{align*}
for a non tangential region $\Gamma_{\beta}(\xi)$ with vertex at the same point $\xi \in \mathbb T$
as $\Gamma_{1/2}(\xi)$\,.
%and $\beta \in (1/2,1)$,  such that $\cup_{\alpha \in %\Gamma_{1/2}(\xi)} \Delta_{\eta}(\alpha)\subseteq \Gamma_{\beta}(\xi)$.
\end{lemma}
\begin{proof}
If $\alpha\in \mathcal{A}$ then
\begin{align*}
|f(\alpha)|^{p}<\frac{\varepsilon}{m(\Delta_{\eta}(\alpha))} \int_{\Delta_\eta(\alpha)} |f(z)|^{p}\ dm(z)\,\,.
\end{align*}
Hence, integrating both sides over $\mathcal{A}\cap \Gamma_{1/2}(\xi)$ and considering a
$\Gamma_{\beta} (\xi)$ as in Remark~\ref{RemarkDisks}
\begin{align*}
	\int_{\mathcal{A}\cap \Gamma_{1/2}(\xi)} |f(\alpha)|^{p}\ \frac{dm(\alpha)}{1-|\alpha|}< &\varepsilon 	\int_{\mathcal{A}\cap \Gamma_{1/2}(\xi)}\left(\frac{1}{m(\Delta_{\eta}(\alpha))}\int_{\Delta_{\eta}(\alpha)} |f(z)|^{p}\ dm(z)\right)\ \frac{dm(\alpha)}{1-|\alpha|}\\
\leq  & \varepsilon 	\int_{\mathcal{A}\cap \Gamma_{1/2}(\xi)}\left(\frac{1}{m(\Delta_{\eta}(\alpha))}\int_{\Gamma_{\beta}(\xi)} |f(z)|^{p}\, \chi_{\Delta_{\eta}(\alpha)}(z)\ dm(z)\right)\ \frac{dm(\alpha)}{1-|\alpha|}\,.
\end{align*}
\\
Remark~\ref{RemarkDisks} also implies that $\chi_{\Delta_\eta(\alpha)}(z)\leq \chi_{\Delta(\alpha,r)}(z)$  when $\frac{2\eta}{1+\eta^2}\leq r<1$.
Therefore, applying Fubini's theorem 
\begin{align*}
	\int_{\mathcal{A}\cap \Gamma_{1/2}(\xi)} |f(\alpha)|^{p}\ \frac{dm(\alpha)}{1-|\alpha|}<\varepsilon 	\int_{\Gamma_{\beta}(\xi)} |f(z)|^{p} \int_{\mathcal{A}\cap \Gamma_{1/2}(\xi)}\left(\frac{\chi_{\Delta(z,r)}(\alpha)}{m(\Delta_\eta(\alpha))} \right)\ \frac{dm(\alpha)}{1-|\alpha|}\ dm(z)\,.
\end{align*}
Remember that
 $$1-|\alpha|\asymp 1-|z|\,\,, \quad \alpha \in\Delta(z,r)$$ 
and  
$$ m(\Delta_{\eta}(\alpha))\asymp (1-|\alpha|)^2 \asymp m(\Delta(z,r))\,\,, \quad \alpha \in\Delta(z,r)\,, $$ 
where the constants involved depend on $\eta$\,. Thus,  we obtain that
\\
\begin{align*}
	\int_{\mathcal{A}\cap \Gamma_{1/2}(\xi)} |f(\alpha)|^{p}\ \frac{dm(\alpha)}{1-|\alpha|}
	%&\lesssim\varepsilon 	\int_{\Gamma_{M'}(\xi)} |f(z)|^{p} \int_{\mathcal{A}\cap \Gamma_{M}(\xi)} \frac{\chi_{\Delta(z,r)}(\alpha)}{(1-|\alpha|)^2} \ dm(\alpha)\ \frac{dm(z)}{1-|z|}\\
%	\\
	& \leq \, C_1 \,\varepsilon \,\int_{\Gamma_{\beta}(\xi)} |f(z)|^{p} \ \frac{dm(z)}{1-|z|}\,.
\end{align*}
\\
\end{proof}

Let $1\leq p<\infty$ and $\lambda\in (0,1)$. We define the set
\begin{align*}
	E_\lambda(\alpha)= \left\{z\in\Delta_\eta(\alpha)\ :\ |f(z)|^p\geq \lambda |f(\alpha)|^p\right\}
\end{align*} 
and the function
\begin{align*}
	B_\lambda f(\alpha) = \frac{1}{m(E_\lambda(\alpha))}\int_{E_\lambda(\alpha)} |f(z)|^p\ dm(z)\,, \quad \quad \alpha \in \mathbb{D}\,.
\end{align*}

The following is the analogue of Lemma~\ref{L1} for $B_\lambda f$\,.
\begin{lemma}\label{L2}
	If $p\geq 1$, $\varepsilon\in (0,1)$, $f\in \mathcal{H}(\D)$, $\lambda\in (0,\frac{1}{2^p})$, and
	\begin{align*}
	B= \left\{\alpha\in\D\ : |f(\alpha)|^{p}<\varepsilon^{2+\frac{2}{p}}B_{\lambda}(f)(\alpha)\right\}.
	\end{align*}
	Then, there is a constant $C_2=C_2(\eta)>0$ such that
	\begin{align*}
	\int_{{B}\cap \Gamma_{1/2}(\xi)} |f(z)|^{p}\ \frac{dm(z)}{1-|z|}\leq \,\varepsilon\, C_2 \int_{\Gamma_{\beta}(\xi)} |f(z)|^{p}\ \frac{dm(z)}{1-|z|}\,\,,
	\end{align*}
for a non tangential region $\Gamma_{\beta}(\xi)$ with vertex at the same $\xi \in \mathbb T$ as $\Gamma_{1/2}(\xi)$\,.
% and $\beta \in (1/2,1)$ such that  $\cup_{\alpha \in \Gamma_{1/2}(\xi)} %\Delta_{\eta}(\alpha)\subset \Gamma_{\beta}(\xi)$.
\end{lemma}
\begin{proof}
Applying Lemma~\ref{L1}
\\
\begin{align*}
\int_{{B}\cap \Gamma_{1/2}(\xi)} |f(z)|^{p}\ \frac{dm(z)}{1-|z|}&=\int_{{B}\cap \Gamma_{1/2}(\xi)\cap \mathcal{A}} |f(z)|^{p}\ \frac{dm(z)}{1-|z|}+\int_{({B}\cap \Gamma_{1/2}(\xi))\setminus\mathcal{A}} |f(z)|^{p}\ \frac{dm(z)}{1-|z|}\\
\\
&\leq C_1(\eta) \,\varepsilon \int_{\Gamma_{\beta}(\xi)} |f(z)|^{p} \ \frac{dm(z)}{1-|z|}+\int_{({B}\cap \Gamma_{1/2}(\xi))\setminus\mathcal{A}} |f(z)|^{p}\ \frac{dm(z)}{1-|z|}.
\end{align*}

We are looking for a similar $\varepsilon$-estimate of the second additive. If $\alpha\in{B}$ then
\begin{align*}
|f(\alpha)|^p\,<\, \frac{\varepsilon^{2+\frac{2}{p}}}{m(E_\lambda(\alpha))}\int_{E_\lambda(\alpha)} |f(z)|^p\ dm(z)\,.
\end{align*}
So, integrating both sides over $({B}\cap \Gamma_{1/2}(\xi))\setminus\mathcal{A}$  we get
\begin{align*}
\int_{({B}\cap \Gamma_{1/2}(\xi))\setminus\mathcal{A}} |f(\alpha)|^{p}\ \frac{dm(\alpha)}{1-|\alpha|}<\varepsilon^{2+\frac{2}{p}} \int_{({B}\cap \Gamma_{1/2}(\xi))\setminus\mathcal{A}}  \frac{1}{m(E_\lambda(\alpha))}\int_{E_\lambda(\alpha)} |f(z)|^p\ dm(z)\ \frac{dm(\alpha)}{1-|\alpha|}\,\,.
\end{align*}
\\
Assume that 
\begin{equation}\label{inclusion}
	\left\{z\in\D\ :\ |z-\alpha|<\frac{\varepsilon^{1/p} \eta (1-|\alpha|)}{2 C}\right\}\subset E_\lambda(\alpha)\,\,,\quad  \forall \alpha \notin \mathcal A\,,
\end{equation}	
where $C$ is a positive constant to be determined later,
then 
\begin{align*}
	m(E_\lambda(\alpha))\geq \,\,\frac{\varepsilon^{2/p} }{4 C^2}\,\, m(\Delta_\eta(\alpha))
	\,\,,\quad  \forall \alpha \notin \mathcal A\,.
\end{align*}
Thus,
\begin{align*}
	\int_{({B}\cap \Gamma_{1/2}(\xi))\setminus\mathcal{A}}& |f(\alpha)|^{p}\ \frac{dm(\alpha)}{1-|\alpha|}\,\\
	& \leq\, 4 C^2\,\varepsilon^2 \int_{({B}\cap \Gamma_{1/2}(\xi))\setminus\mathcal{A}}  \frac{1}{m(\Delta_{\eta}(\alpha))}\int_{E_\lambda(\alpha)} |f(z)|^p\ dm(z)\ \frac{dm(\alpha)}{1-|\alpha|}\\
	&\leq \, 4 C^2\,\varepsilon^2 \int_{({B}\cap \Gamma_{1/2}(\xi))\setminus\mathcal{A}}  \frac{1}{m(\Delta_{\eta}(\alpha))}\int_{\Delta_{\eta}(\alpha)} |f(z)|^p\ dm(z)\ \frac{dm(\alpha)}{1-|\alpha|}\\
	& \leq \,  4 C^2 \, \epsilon\, \int_{({B}\cap \Gamma_{1/2}(\xi))\setminus\mathcal{A}} |f(\alpha)|^{p}\ \frac{dm(\alpha)}{1-|\alpha|}\\
	&\,\leq \,  4 C^2 \, \varepsilon \,\int_{\Gamma_{\beta}(\xi)} |f(\alpha)|^{p}\ \frac{dm(\alpha)}{1-|\alpha|}
\end{align*}
and the lemma is proved.

To complete the picture we proceed to the proof of inclusion (\ref{inclusion})\,. Let $\alpha \in \mathbb{D} $\, and $z\in \Delta_{\frac{\eta}{4}}(\alpha)$\,.
%the disc $|z-\alpha|<\frac{\eta}{4}(1-|\alpha|)$. 
By Cauchy integral formula, we have
\begin{align}\label{CauchyIF1}
|f(z)-f(\alpha)|&=\frac{1}{2\pi}\left|\int_{|w-\alpha|=\frac{\eta}{2} (1-|\alpha|)} f(w) \left(\frac{1}{w-z}-\frac{1}{w-\alpha}\right)\ dw\right|\nonumber\\
%&=\frac{1}{2\pi}\left|\int_{|w-\alpha|=\frac{\eta}{2}(1-|\alpha|)} f(w) \left(\frac{z-\alpha}{(w-z)(w-\alpha)}\right)\ dw\right|\nonumber\\
& \leq \frac{1}{2\pi} \int_{|w-\alpha|=\frac{\eta}{2}(1-|\alpha|)} |f(w)| \left|\frac{z-\alpha}{(w-z)(w-\alpha)}\right|\ |dw| \,.
\end{align}
By subhamonicity, for $w\in \overline{\Delta_{\frac{\eta}{2}}(\alpha)}$, we have
%$|w-\alpha|=\frac{\eta}{2}(1-|\alpha|)$, we have
\begin{align*}
|f(w)|\leq \frac{C'}{m(\Delta_\eta(\alpha))} \int_{\Delta_{\eta}(\alpha)} |f(u)| \ dm(u)
\end{align*}
where $C'$ is an absolute positive constant.
 Applying that in (\ref{CauchyIF1})
\begin{align*}
	|f(z)-f(\alpha)|&\leq \frac{4 C'\,|z-\alpha|}{ \eta (1-|\alpha|)}\,\frac{1}{m(\Delta_\eta(\alpha))}\int_{\Delta_\eta(\alpha)} |f(u)|\ dm(u)\,.
\end{align*}
Now, if we consider $z\in \Delta_{\frac{\varepsilon^{\frac{1}{p}} \eta}{2C}}(\alpha)$  for a $C$ large enough then
\begin{align*}
	|f(z)-f(\alpha)|&\leq \frac{\varepsilon^\frac{1}{p}}{2 m(\Delta_\eta(\alpha))}\int_{\Delta_\eta(\alpha)} |f(u)|\ dm(u).
\end{align*}
Therefore,  if $\alpha\notin\mathcal{A}$ then it follows that
\begin{align*}
|f(z)|\geq |f(\alpha)|-|f(\alpha)-f(z)|\geq \frac{1}{2} |f(\alpha)|\,.
\end{align*}
So that,
\begin{align*}
	|f(z)|^{p}\geq \frac{|f(\alpha)|^p}{2^p}>\lambda |f(\alpha)|^p
\end{align*}
and this implies the strength of inclusion (\ref{inclusion})\,.
\\
\end{proof}

\begin{remark}\label{R2}
	Notice that for $f\in \mathcal{H}(\D)$, $\alpha\in \D$ and $\lambda\in (0,1)$ we have that
\begin{align*}
\frac{m(E_\lambda(\alpha))}{m(\Delta_\eta(\alpha))}\geq \frac{\log(1/\lambda)}{\log\left(\frac{B_{\lambda}(f)(\alpha)}{|f(\alpha)|^{p}}\right)+\log(1/\lambda)}.
\end{align*}
Now if $\alpha\in \D\setminus B$,\, $\varepsilon\in (0,1)$,\, $\lambda\in (0,1/2^p)$ it follows that $$\frac{B_{\lambda}(f)(\alpha)}{|f(\alpha)|^{p}}\leq \frac{1}{\varepsilon^{2+\frac{2}{p}}}\,\,.$$ 
If $\delta \in (0,1)$ then choosing $\lambda<\varepsilon^{\frac{2}{\delta}(2+\frac{2}{p})}$
we obtain
\begin{align*}
	\frac{m(E_\lambda(\alpha))}{m(\Delta_\eta(\alpha))}> \frac{(2/\delta)
\log(1/\varepsilon^{2+\frac{2}{p}})}{\log\left(1/\varepsilon^{2+\frac{2}{p}}\right)+(2/\delta)\log\left(1/\varepsilon^{2+\frac{2}{p}}\right)}>1-\frac{\delta}{2},
\end{align*}
that is,
\begin{align*}
m(E_\lambda(\alpha))>\left(1-\frac{\delta}{2}\right)m(\Delta_\eta(\alpha)).
\end{align*}
\end{remark}

Now, we state the last auxiliary Lemma. In order to prove it we use an argument based on  
the family of functions $\{f_{\alpha}\}_{\alpha \in \mathbb D} \subset AT^q_p\,,\, p,q \in [1,\infty),$ 
of Proposition~\ref{Sg}\,. Recall that 
$$ \|f_{\alpha}\|_{T_{p}^{q}} \,\asymp\, 1 \,, \quad \forall \alpha \in \mathbb D\,.$$
%These are the functions
%\begin{align*}
%	f_{\alpha}(z)=\frac{(1-|\alpha|^2)^{\beta-\frac{1}{p}-\frac{1}{q}}}{(1-\overline{\alpha}z)^{\beta}}\,,\quad z\in\D\,,
%\end{align*}
%where $\beta> (\frac{1}{p}+\frac{1}{q})$\,. {\color {} See ....................}
%We will make use of the following in the proof of the necessary condition of the main result. 

\begin{lemma}\label{Basic Lemma}
	Let $0<p,q<\infty$. Given $\varepsilon>0$, there exist $\eta\in (0,1)$ such that for all $\alpha\in\D$ there is a function $f_{\alpha}\in AT_p^q$ such that
	\begin{enumerate}
		\item $\|f_\alpha\|_{T_p^q}\asymp 1$, and
		\item $\|f_\alpha \chi_{\Delta^{c}(\alpha,\eta)}\|_{T_p^q}\leq \varepsilon$.
	\end{enumerate}
 \end{lemma}
\begin{proof}
	Fix $\lambda>\max\{1,p/q\}$. First of all, we consider the family $\{f_\alpha\}$ of the functions 
	$$f_{\alpha}(z)=\frac{(1-|\alpha|^2)^{\gamma-\frac{1}{p}-\frac{1}{q}}}{(1-\overline{\alpha}z)^{\gamma}},\, z\in\D,$$
	where $\gamma>\frac{\lambda+2}{p}$. As we have seen previously, these functions satisfies \emph{(1)}. Let us see now that, for a given $\varepsilon>0$, there is $\eta\in (0,1)$ such that 
	\begin{align*}
\mathbb I & = \int_{\T} \left(\int_{\Gamma_{\frac 12}(\xi)\setminus \Delta(\alpha,\eta)} |f_{\alpha}(z)|^{p}\ \frac{dm(z)}{1-|z|}\right)^{q/p}\ |d\xi|\leq \varepsilon^{q}.
	\end{align*}
	
	Observe that 
	\begin{align*}
		\mathbb I & = \int_{\mathbb T} \left(\int_{\Gamma_{\frac 12}(\xi)\setminus \Delta(\alpha,\eta)} |f_{\alpha}(z)|^{p}\ \frac{dm(z)}{1-|z|}\right)^{q/p}\ |d\xi| \\
		& = \int_{\mathbb T} \left(\int_{\Gamma_{\frac 12}(\xi)} \chi_{\Delta^c(\alpha,\eta)}(z) |f_{\alpha}(z)|^{p}\ \frac{dm(z)}{1-|z|}\right)^{q/p}\ |d\xi|\,.
	\end{align*}
Since $\lambda > \max\{1,\, p/q\}$ and $\chi_{\Gamma_{\frac{1}{2}}(\xi)}(z)\lesssim \frac{1-|z|^2}{|1-\overline{\xi} z|}$, we have that
\begin{align*}
	\mathbb I   \lesssim \int_{\mathbb T} \left(\int_{\mathbb {D}} \left(\frac{1-|z|^2}{|1-\overline{\xi}z|} \right)^{\lambda}\chi_{\Delta^c(\alpha,\eta)}(z) |f_{\alpha}(z)|^{p}\ \frac{dm(z)}{1-|z|}\right)^{q/p}\ |d\xi|\,.
\end{align*}
Applying in the inner integral the  change of  variable
$$ z \mapsto \phi_{\alpha}(z) ,$$ 
where $ \phi_{\alpha}(z)  = \frac {\alpha - z}{1-\overline{\alpha}z}\,, $
we get that
\begin{align*}
	\mathbb I & \lesssim \int_{\mathbb T} \left(\int_{\mathbb {D}\setminus \Delta(0,\eta)} \frac{(1-|\phi_\alpha (z)|^2)^{\lambda-1}}{|1-\overline{\xi}\phi_\alpha (z)|^{\lambda}}  |f_{\alpha}(\phi_\alpha (z))|^{p}\ |\phi_{\alpha}'(z)|^2\ dm(z) \right)^{q/p}\ |d\xi|\\
	& =\int_{\mathbb T} \frac{1}{|1-\overline{\alpha} \xi|^{\lambda q/p}} \left(\int_{\mathbb {D}\setminus \Delta(0,\eta)} \frac{(1-|\phi_\alpha (z)|^2)^{\lambda-1}\,|1-\overline{\alpha}z|^{\lambda}}{|1-\overline{\phi_{\alpha}(\xi)} z|^{\lambda}}  |f_{\alpha}(\phi_\alpha (z))|^{p}\ |\phi_{\alpha}'(z)|^2\ dm(z) \right)^{q/p}\ |d\xi|\,.
\end{align*}
Since 
 $$1-|\phi_{\alpha}(z)|^2 = |\phi_{\alpha}'(z)|\, (1-|z|^2), \quad \alpha,z \in \mathbb D, $$
\begin{align*}
	\mathbb I & \lesssim \int_{\mathbb T} \frac{1}{|1-\overline{\alpha} \xi|^{\lambda q/p}} \left(\int_{\mathbb {D}\setminus \Delta(0,\eta)} \frac{(1-|z|^2)^{\lambda-1}\,\,|1-\overline{\alpha}z|^{\lambda}}{|1-\overline{\phi_{\alpha}(\xi)} z|^{\lambda}}  |f_{\alpha}(\phi_\alpha (z))|^{p}\ |\phi_{\alpha}'(z)|^{\lambda +1}\ dm(z) \right)^{q/p}\ |d\xi|\,.
\end{align*}
Verifying that 
$$
\,|1-\overline{\alpha}z|^{\lambda} \, |f_{\alpha}(\phi_{\alpha}(z))|^p\, |\phi'_{\alpha}(z)|^{\lambda +1} \, = \, (1-|\alpha|^2)^{\lambda - \frac pq} |1 - \overline{\alpha}z|^{\gamma p -\lambda -2}
$$
we get
\begin{align*}
	\mathbb I & \lesssim (1-|\alpha|^2)^{\lambda\,\frac qp -1} \int_{\mathbb T} \frac{1}{|1-\overline{\alpha} \xi|^{\lambda q/p}} \left(\int_{\mathbb {D}\setminus \Delta(0,\eta)} \frac{(1-|z|^2)^{\lambda-1}}{|1-\overline{\phi_{\alpha}(\xi)} z|^{\lambda}}  |1 - \overline{\alpha} z|^{\gamma p -\lambda -2}\ dm(z) \right)^{q/p}\ |d\xi|\,.
\end{align*}
Since $\gamma >  \frac {\lambda + 2}{p}, $
$$|1 - \overline{\alpha} z|^{\gamma p -\lambda -2} \leq 2^{\gamma p -\lambda -2}\,.$$
 Thus,
\begin{align*}
		\mathbb I & \lesssim (1-|\alpha|^2)^{\lambda\,\frac qp -1}  \int_{\mathbb T} \frac{1}{|1-\overline{\alpha} \xi|^{\lambda q/p}} \left(\int_{\mathbb {D}\setminus \Delta(0,\eta)} \frac{(1-|z|^2)^{\lambda-1}}{|1-\overline{\phi_{\alpha}(\xi)} z|^{\lambda}} \ dm(z) \right)^{q/p}\ |d\xi|\\
		& = (1-|\alpha|^2)^{\lambda\,\frac qp -1}  \int_{\mathbb T} \frac{1}{|1-\overline{\alpha} \xi|^{\lambda q/p}} \left(\int_{\eta}^1 \ \frac {1}{\pi} \int_0^{2\pi} \frac{1}{|1-\overline{\phi_{\alpha}(\xi)}r e^{i\theta}|^{\lambda}} \ d\theta\,(1- r^2)^{\lambda-1}\, r\, dr \right)^{q/p}\ |d\xi|\,.
 \end{align*}
Now, a standrad argument based on Parseval's identity, Stirling's formula and the fact that $|\phi_{\alpha}(\xi)|=1=|\xi|$ implies
\begin{align*}
\mathbb I & \lesssim (1-|\alpha|^2)^{\lambda\,\frac qp -1}  \int_{\mathbb T} \frac{1}{|1-\overline{\alpha} \xi|^{\lambda q/p}}\,\ |d\xi| \,(1-\eta)^{q/p}\,.
\end{align*}
Employing the same argument one more time, we obtain that
\begin{align*}
\mathbb I & \lesssim (1-\eta)^{q/p}\,.
\end{align*}
Therefore, we conclude that, for a given $\varepsilon>0$, there is $\eta\in (0,1)$ such that \emph{(2)} holds.
\end{proof}

%{\color{} ¿SE PUEDE QUITAR? \begin{remark}\label{best}
%Let $\eta \in (0,1)$. If $\alpha \in \D$ then
%$$
%\xi_0 \in I_{\alpha}=\{\xi\in\T\ : \Delta (\alpha,\eta)\cap \Gamma_{\frac 12}(\xi)\neq \emptyset\}
%$$
% if and only if 
%$\a \in \Gamma_{\eta'}(\xi_0)$ where $1 - \eta' \asymp 1-\eta$\,. Moreover, in the case $R_0 = \frac{1+\eta'}{2} \leq |\alpha| < 1$ it is true that
%$$ |I_{\alpha}| \asymp \frac{1-|\alpha|}{(1-\eta')^{1/2}}\asymp \frac{1-|\alpha|}{(1-\eta)^{1/2}} $$
%\end{remark}}

Now we are in position to present the complete version of our main result.
\begin{theorem}\label{MainTheorem}
	Let $1\leq p,q<+\infty$ and $G$ be a measurable subset of  $\D$. The following assertions are equivalent:
	\begin{enumerate}
	\item[(a)] There is a $\beta \in (0,1)$ and a constant $K>0$ such that
	\begin{align*}
	\left(\int_{\T} \left(\int_{\Gamma_{\beta}(\xi)\cap G} |f(z)|^{p}\ \frac{dm(z)}{1-|z|}\right)^{q/p}\ |d\xi|\right)^{1/q}\geq \, K \, \|f\|_{T_{p}^{q}},
	\end{align*}
	for every $f\in AT_{p}^{q}$\,.
	\item[(b)] There exist an $\eta\in (0,1)$ and a $\delta>0$ such that
	
	\begin{align*}
	m(G\cap \Delta(\alpha,\eta))\geq \delta\, m(\Delta(\alpha,\eta))
	\end{align*}
	for $\alpha \in \D$\,.
	\item[(c)] There exist an $\eta\in (0,1)$ and a $\delta>0$ such that
	
	\begin{align*}
		m(G\cap \Delta_\eta(\alpha))\geq \,\delta\, m(\Delta_\eta(\alpha))
	\end{align*}
	for $\alpha \in \D$\,.
	\end{enumerate}
\end{theorem}

\begin{proof}
	$ \emph{(b)} \Leftrightarrow \emph{(c)} :$ The proof is the same as in \cite{Luecking_1981}\,.\\
$\emph{(c)} \Rightarrow \emph{(a)}\, :\,$ Assume \emph{(c)} holds. 
We want to prove that there is a $\beta \in (0,1)$ and a constant $K>0$ such that
\begin{align*}
	\left(\int_{\T} \left(\int_{\Gamma_{\beta}(\xi)\cap G} |f(z)|^{p} \ \frac{dm(z)}{1-|z|}\right)^{q/p} |d\xi|\right)^{1/q} \geq \,K \,\left(\int_{\T} \left(\int_{\Gamma_{1/2}(\xi)} |f(z)|^{p} \ \frac{dm(z)}{1-|z|}\right)^{q/p} |d\xi|\right)^{1/q}
\end{align*}
for all $f\in AT_{p}^{q}$.

Consider an $\varepsilon \in (0,1)$ and a 
$\lambda < \min\{ 1/2^p,\, \varepsilon^{\frac{2}{\delta}(2+2/p)}\}\,.$ By Remark ~\ref{R2} we obtain that
\begin{align*}
m(G\cap E_{\lambda}(\alpha))&=m(G\cap \Delta_\eta(\alpha))-m(G\cap (\Delta_\eta(\alpha)\setminus E_{\lambda}(\alpha)))\\
&\geq \delta m(\Delta_\eta(\alpha))-m(\Delta_\eta(\alpha)\setminus E_{\lambda}(\alpha))\\
&=\delta m(\Delta_\eta(\alpha))-m(\Delta_\eta(\alpha))+m(E_{\lambda}(\alpha))\\
&\geq \delta m(\Delta_\eta(\alpha)) -m(\Delta_\eta(\alpha)) +\left(1-\frac{\delta}{2}\right)m(\Delta_\eta(\alpha))\\
&=\frac{\delta}{2} m(\Delta_\eta(\alpha))
\end{align*}
for any $\alpha \in \mathbb D\setminus B$, where $B$ is the set in Lemma~\ref{L2}\,.

Let a pseudohyperbolic radius $ r=r(\eta) \in (0,1)$  as in Remark~\ref{RemarkDisks}.  We can find a
 $\beta\in (1/2,1)$ such that $\cup_{\alpha \in \Gamma_{1/2}(\xi)} \Delta(\alpha,r)\subset \Gamma_{\beta}(\xi)$\,.

 If $\alpha\in \Gamma_{1/2}(\xi)\setminus {B}$ then using the fact that $E_\lambda(f)(\alpha)\subset \Delta_\eta(\alpha)\subset \Gamma_{\beta}(\xi)$
\begin{align*}
&\frac{1}{m(\Delta_\eta(\alpha))} \int_{G\cap \Gamma_{\beta}(\xi)} \chi_{\Delta_\eta(\alpha)}(z) |f(z)|^{p} \frac{dm(z)}{1-|z|}\\
&\qquad\geq \frac{\delta}{2} \frac{1}{m(G \cap E_\lambda(\alpha))} \int_{G\cap \Gamma_{\beta}(\xi)} \chi_{\Delta_\eta(\alpha)}(z) |f(z)|^{p} \frac{dm(z)}{1-|z|}\\
&\qquad\geq \frac{\delta}{2} \frac{1}{m(G \cap E_\lambda(\alpha))} \int_{G\cap E_\lambda(\alpha)} \chi_{\Delta_\eta(\alpha)}(z) |f(z)|^{p} \frac{dm(z)}{1-|z|}\\
&\qquad\geq C_3 \frac{\delta\lambda}{2}\frac{|f(\alpha)|^{p}}{(1-|\alpha|)}\,.
\end{align*}
%\gtrsim
The constant $C_3$ depends on $\eta$.
Integrating over the set $\Gamma_{1/2}(\xi)\setminus {B}$ and 
applying Fubini's theorem we obtain
%\begin{align*}
%\int_{\Gamma_{\beta}(\xi)\setminus {B}} \frac{1}{m(\Delta_\eta(\alpha))} \int_{G\cap %\Gamma_{\beta'}(\xi)} \chi_{\Delta_\eta(\alpha)}(z) |f(z)|^{p} \frac{dm(z)}{1-|z|}\ %dm(\alpha)\geq \frac{\delta\lambda}{2} \int_{\Gamma_{M}(\xi)\setminus {B}} |f(\alpha)|^{p}\ %\frac{dm(\alpha)}{1-|\alpha|}.
%\end{align*}

%By Fubini's theorem, we obtain
\begin{align*}
\int_{G\cap \Gamma_{\beta}(\xi)}|f(z)|^{p} \int_{\Gamma_{1/2}(\xi)\setminus {B}}  \frac{\chi_{\Delta(z,r)}(\alpha)}{m(\Delta_\eta(\alpha))}\ dm(\alpha) \frac{dm(z)}{1-|z|} \geq \, C_3 \,\frac{\delta\lambda}{2} \int_{\Gamma_{1/2}(\xi)\setminus {B}} |f(\alpha)|^{p}\ \frac{dm(\alpha)}{1-|\alpha|}.
\end{align*}

Using arguments we have seen before we conclude that there is a positive constant $C_4=C_4(\eta)$
 such that
$$
\int_{\Gamma_{1/2}(\xi)\setminus {B}}  \frac{\chi_{\Delta(z,r)}(\alpha)}{m(\Delta(\alpha,\eta))}\ dm(\alpha)\leq C_4
$$
for any $z\in \D$. Therefore, 
\begin{align*}
 \int_{G\cap \Gamma_{\beta}(\xi)}|f(z)|^{p}\  \frac{dm(z)}{1-|z|} \geq \, C_5\, \frac{\delta\lambda}{2} \int_{\Gamma_{1/2}(\xi)\setminus {B}}  |f(\alpha)|^{p}\ \frac{dm(\alpha)}{1-|\alpha|}.
\end{align*}
Applying Lemma~\ref{L2}, it follows that 
\begin{align*}
&\int_{G\cap \Gamma_{\beta}(\xi)}|f(z)|^{p}\  \frac{dm(z)}{1-|z|} \\
&\qquad\geq \,C_5\,\frac{\delta\lambda}{2} \int_{\Gamma_{1/2}(\xi)}  |f(\alpha)|^{p}\ \frac{dm(\alpha)}{1-|\alpha|}-\,C_5\,\frac{\delta\lambda}{2} \int_{\Gamma_{1/2}(\xi)\cap {B}}  |f(\alpha)|^{p}\ \frac{dm(\alpha)}{1-|\alpha|}\\
&\qquad\geq \,C_5\,\frac{\delta\lambda}{2} \int_{\Gamma_{1/2}(\xi)}  |f(\alpha)|^{p}\ \frac{dm(\alpha)}{1-|\alpha|}-\,C_5\,\frac{\delta\lambda\varepsilon C_2}{2} \int_{\Gamma_{\beta}(\xi)}  |f(\alpha)|^{p}\ \frac{dm(\alpha)}{1-|\alpha|}\,.\\
\end{align*}

So that,   
\begin{align*}
\left(\int_{G\cap \Gamma_{\beta}(\xi)}|f(z)|^{p}\  \frac{dm(z)}{1-|z|}\right)^{1/p}&+ \left(\,C_6\,\frac{\delta\lambda\varepsilon}{2}\right)^{1/p}\left(\int_{\Gamma_{\beta}(\xi)}  |f(\alpha)|^{p}\ \frac{dm(\alpha)}{1-|\alpha|}\right)^{1/p}\\
&\geq \left(\,C_5\,\frac{\delta\lambda}{2}\right)^{1/p}\left(\int_{\Gamma_{1/2}(\xi)}  |f(\alpha)|^{p}\ \frac{dm(\alpha)}{1-|\alpha|}\right)^{1/p}.
\end{align*}

By means of Minkowski's inequality we get that 
\begin{align*}
&\left(\int_{\T}\left(\int_{G\cap \Gamma_{\beta}(\xi)}|f(z)|^{p}\  \frac{dm(z)}{1-|z|}\right)^{q/p}\ |d\xi|\right)^{1/q}\\
&\quad +\left(\,C_6\,\frac{\delta\lambda\varepsilon }{2}\right)^{1/p}\left(\int_{\T}\left(\int_{\Gamma_{\beta}(\xi)}  |f(\alpha)|^{p}\ \frac{dm(\alpha)}{1-|\alpha|}\right)^{q/p}\ |d\xi|\right)^{1/q}\\
&\quad \geq \left(\,C_5\,\frac{\delta\lambda}{2}\right)^{1/p}\left(\int_{\T}\left(\int_{\Gamma_{1/2}(\xi)}  |f(\alpha)|^{p}\ \frac{dm(\alpha)}{1-|\alpha|}\right)^{q/p}\ |d\xi|\right)^{1/q}.
\end{align*}

According to  Lemma~\ref{estimate 1}, the $AT^q_p$-norm can be equivalently expressed by any cone-like region. So, if  we choose  $\varepsilon$ small enough then we end up that there is a constant $K>0$ such that 
%\begin{align*}
%&\left(\int_{\T}\left(\int_{G\cap \Gamma_{M}(\xi)}|f(z)|^{p}\  %\frac{dm(z)}{1-|z|}\right)^{q/p}\ |d\xi|\right)^{1/q}\\
%&\quad +\left(\frac{\delta\lambda\varepsilon C_2 %C_4}{2C_3}\right)^{1/p}\left(\int_{\T}\left(\int_{\Gamma_{M}(\xi)}  |f(\alpha)|^{p}\ %\frac{dm(\alpha)}{1-|\alpha|}\right)^{q/p}\ |d\xi|\right)^{1/q}\\
%&\quad \geq %\left(\frac{\delta\lambda}{2C_3}\right)^{1/p}\left(\int_{\T}\left(\int_{\Gamma_{M}(\xi)}  %|f(\alpha)|^{p}\ \frac{dm(\alpha)}{1-|\alpha|}\right)^{q/p}\ |d\xi|\right)^{1/q}.
%\end{align*}

%Therefore, choosing $\varepsilon$ small enough, we conclude that there a constant $K>0$ such %that 
\begin{align*}
\left(\int_{\T}\left(\int_{G\cap \Gamma_{\beta}(\xi)}|f(z)|^{p}\  \frac{dm(z)}{1-|z|}\right)^{q/p}\ |d\xi|\right)^{1/q}\geq K \|f\|_{T_{p}^{q}}.
\end{align*}

$\emph{(a)}\Rightarrow \emph{(b)} :$ Before starting with the proof, we need to show the following estimate:
\begin{align*}
\left(\int_{\T}\left(\int_{\Gamma_{\beta}(\xi)} \chi_{G\cap \Delta(\alpha,\eta)}(z)|f_\alpha(z)|^{p}\  \frac{dm(z)}{1-|z|}\right)^{q/p}\ |d\xi|\right)^{p/q} \lesssim \frac{m(G\cap \Delta(\alpha,\eta))}{m( \Delta(\alpha,\eta))},
\end{align*}
where $f_\alpha$ is the function considered in Lemma~\ref{Basic Lemma}.
Applying Lemma~\ref{estimate 1} and properties of pseudo-hyperbolic disks (see, e.g., \eqref{psh compar} and \eqref{psh area}), it follows that
\begin{align}
&\left(\int_{\T}\left(\int_{\Gamma_{\beta}(\xi)}\chi_{G\cap \Delta(\alpha,\eta)}(z)|f_\alpha(z)|^{p}\  \frac{dm(z)}{1-|z|}\right)^{q/p}\ |d\xi|\right)^{p/q}\nonumber\\
&\qquad \asymp \frac{(1-|\alpha|)^{1-p/q}}{m(\Delta(\alpha,\eta))} \left(\int_{\T}\left(\int_{ \Gamma_{\beta}(\xi)} \chi_{G\cap \Delta(\alpha,\eta)}(z)\  \frac{dm(z)}{1-|z|}\right)^{q/p}\ |d\xi|\right)^{p/q}\nonumber\\
&\qquad \asymp \frac{(1-|\alpha|)^{-p/q}}{m(\Delta(\alpha,\eta))} \left(\int_{\T}\left(\int_{G\cap \Delta(\alpha,\eta)} \frac{(1-|z|)^\lambda}{|1-z\overline{\xi}|^\lambda} \ dm(z)\right)^{q/p}\ |d\xi|\right)^{p/q}, \label{ineq:part1}
\end{align}
where $\lambda>\max\{1,p/q\}$. Notice that, by means of rotations, we can assume the pseudo-hyperbolic disk centered at $\alpha \in (0,1)$ without loss of generality.
Applying the change of variable $ z \mapsto \phi_{\alpha}(z)$, where $ {\phi_{\alpha}(z)  = \frac {\alpha - z}{1-{\alpha}z}}\, $, it follows
\begin{align*}
&\left(\int_{\T}\left(\int_{G\cap \Delta(\alpha,\eta)} \frac{(1-|z|^2)^\lambda}{|1-z\overline{\xi}|^\lambda} \ dm(z)\right)^{q/p}\ |d\xi|\right)^{p/q}\\
&\qquad=\left(\int_{\T}\left(\int_{\phi_\alpha(G\cap \Delta(\alpha,\eta))} \frac{(1-|\phi_{\alpha}(z)|^2)^\lambda}{|1-\overline{\xi}\phi_{\alpha}(z)|^\lambda} |\phi'_{\alpha}(z)|^2 \ dm(z)\right)^{q/p}\ |d\xi|\right)^{p/q}\\
&\qquad=\left(\int_{\T}\frac{(1-|\alpha|^2)^{\lambda q/p}}{|1-\alpha \xi|^{\lambda q/p}}\left(\int_{\phi_{\alpha}(G\cap \Delta(\alpha,\eta))} \frac{(1-|z|^2)^\lambda |\phi'_{\alpha}(z)|^2}{|1-\alpha z|^{\lambda}|1-z\overline{\phi_{\alpha}(\xi)}|^\lambda}  \ dm(z)\right)^{q/p}\ |d\xi|\right)^{p/q}.
\end{align*}
Applying the same change of variable again, we obtain
\begin{align*}
&\left(\int_{\T}\left(\int_{G\cap \Delta(\alpha,\eta)} \frac{(1-|z|^2)^\lambda}{|1-z\overline{\xi}|^\lambda} \ dm(z)\right)^{q/p}\ |d\xi|\right)^{p/q}\\
&\qquad =\left(\int_{\T}\frac{(1-|\alpha|^2)^{\lambda q/p}}{|1-\alpha \xi|^{\lambda q/p}}\left(\int_{G\cap \Delta(\alpha,\eta)} \frac{(1-|\phi_{\alpha}(z)|^2)^\lambda }{|1-\alpha \phi_{\alpha}(z)|^{\lambda}|1-\phi_{\alpha}(z)\overline{\phi_{\alpha}(\xi)}|^\lambda}  \ dm(z)\right)^{q/p}\ |d\xi|\right)^{p/q}\\
&\qquad \lesssim \left(\int_{\T}\frac{1}{|1-\alpha \xi|^{\lambda q/p}}\left(\int_{G\cap \Delta(\alpha,\eta)} |1-\alpha z |^{\lambda}   \ dm(z)\right)^{q/p}\ |d\xi|\right)^{p/q}\\
&\qquad \asymp m(G\cap \Delta(\alpha,\eta))(1-\alpha)^{\lambda} \left(\int_{\T}  \frac{1}{|1-\alpha\xi|^{\lambda q/p}}\ |d\xi|\right)^{p/q}. 
\end{align*}
By a standard argument, the same as the one used at the end of the proof of Lemma~\ref{Basic Lemma}, we get that
\begin{align}
\left(\int_{\T}\left(\int_{G\cap \Delta(\alpha,\eta)} \frac{(1-|z|^2)^\lambda}{|1-z\overline{\xi}|^\lambda} \ dm(z)\right)^{q/p}\ |d\xi|\right)^{p/q}&\lesssim 
 m(G\cap \Delta(\alpha,\eta))(1-|\alpha|)^{p/q}.\label{ineq:part2}
\end{align}
Thus, putting together \eqref{ineq:part1} and \eqref{ineq:part2} we have obtained that
\begin{align*}
\|f \chi_{G\cap \Delta(\alpha,\eta)}\|_{T_p^q}\lesssim \left(\frac{m(G\cap\Delta(\alpha,\eta))}{m(\Delta(\alpha,\eta))}\right)^{1/p},
\end{align*}
for $\alpha\in\D$ and $\eta\in (0,1)$.

Now, we continue with the proof of $\emph{(a)}\Rightarrow \emph{(b)}$. Given $\varepsilon< \frac{K C_1}{2}$. By Lemma~\ref{Basic Lemma}, there is $\eta\in (0,1)$ such that for all pseudo-hyperbolic disks $\Delta(\alpha,\eta)$ there is a function $f_\alpha$ satisfying the following conditions:
\begin{enumerate}
	\item There are constants $C_1,C_2>0$ such that $C_1\leq \|f \|_{T_p^q}\leq C_2$, and
	\item $\|f \chi_{\Delta(\alpha,\eta)^{c}}\|_{T_p^q}\leq\varepsilon$.
\end{enumerate}
 So that, applying \emph{(a)}, we have
\begin{align*}
\left(\frac{m(G\cap\Delta(\alpha,\eta))}{m(\Delta(\alpha,\eta))}\right)^{1/p}&\gtrsim \|f \chi_{G\cap \Delta(\alpha,\eta)}\|_{T_p^q} \geq \|f \chi_{G}\|_{T_p^q}-\|f \chi_{\Delta(\alpha,\eta)^{c}}\|_{T_p^q}\\
&\geq K\|f \|_{T_p^q}-\varepsilon\geq K C_1-\varepsilon\geq \frac{K C_1}{2}.
\end{align*}   
Therefore, we have proved that there is an $\eta\in (0,1)$ and a $\delta>0$ such that 
$$m(G\cap \Delta(\alpha,\eta))\geq \delta m(\Delta(\alpha,\eta))$$
for every $\alpha\in\D$.

\end{proof}

All the above lead to the following statement.
\begin{theorem} 
Let $1\leq p,q<\infty$ and $g\in \mathcal{B}$. The operator  $T_g: RM(p,q) \rightarrow RM(p,q)$  has closed range if and only if there are constants $\delta, c>0$ and $\eta\in(0,1)$ such that

\begin{align}\label{Tg condition}
m(G_c\cap \Delta(\alpha,\eta))\geq \delta \,m(\Delta(\alpha,\eta)), 
\end{align}
for every $\alpha\in\D$, where $G_c=\{z\in\D\ :\ |g'(z)|(1-|z|)>c\}$.
\end{theorem}
\begin{proof}
	Assume first that there are a $c>0$, an $\eta\in (0,1)$ and a $\delta>0$ such that
	\begin{align*}
		m(G_c\cap \Delta(\alpha,\eta))\geq \delta m(\Delta(\alpha,\eta))\,,
	\end{align*}
	for every $\alpha\in\D$\,.
		 We employ Theorem~\ref{MainTheorem} and \cite[Proposition 3.1]{Aguilar-Galanopoulos} to get that there is a $\beta \in (0,1)$ and a constant $ K'>0$ such that 
	\begin{align*}
\left\{\int_{\T} \left(\int_{\Gamma_{\beta}(\xi)\cap G_c} |f(z)|^{p}\ \frac{dm(z)}{1-|z|}\right)^{q/p}\ |d\xi|\right\}^{1/q} &
\geq  K' \,\|f \|_{RM(p,q)}
	\end{align*}
for every $ f \in RM(p,q)$.

Applying the Littlewood-Paley type inequalities, \cite[Proposition 3.1]{Aguilar-Galanopoulos} and that 
$ T_g(f)(0) =0$, we get that 
\begin{align*}
	\|T_g (f)\|_{RM(p,q)}
	&\asymp  \left \{ \int_{\mathbb T} \left(  \int_{\Gamma_{\beta}(\xi)} |f(z)|^{p} \, |g'(z)|^p \, (1-|z|)^p\,\frac{dm(z)}{1-|z|}\right)^{q/p} \right \}^{\frac 1q} \\
	& \geq \left \{ \int_{\mathbb T} \left(  \int_{\Gamma_{\beta}(\xi)\cap G_c} |f(z)|^{p} \, |g'(z)|^p \, (1-|z|)^p\,\frac{dm(z)}{1-|z|}\right)^{q/p} \right \}^{\frac 1q} \\
	&\geq \,c\, \left \{ \int_{\mathbb T} \left(  \int_{\Gamma_{\beta}(\xi)\cap G_c} |f(z)|^{p} \, \frac{dm(z)}{1-|z|}\right)^{q/p} \right \}^{\frac 1q}\,.
\end{align*}

Combining the above, we conclude that there is a constant $C>0$ such that
\begin{equation}\label{useeq1}
	\|T_g (f)\|_{RM(p,q)} \geq  C \,\|f \|_{RM(p,q)}
\end{equation}
for every $f\in RM(p,q)$.

Now, assume there is a constant $C>0$ such that \eqref{useeq1} holds.
 Due to the Littlewood-Paley type inequalities, it is true that
 \begin{align*}
 & \int_{\mathbb T} \left( \int_{\Gamma_{\frac 12}(\xi)} |f(z)|^{p} \, |g'(z)|^p \,(1-|z|)^p\,\frac{dm(z)}{1-|z|}\right)^{q/p}  |d\xi|  \\
& \qquad \geq C' \int_{\T}  \left(\int_{\Gamma_{\frac 12}(\xi)} |f(z)|^{p}\ \frac{dm(z)}{1-|z|}\right)^{q/p} |d\xi|
\end{align*}
  for every $f\in RM(p,q)$.  
  Notice that
  \begin{align*}
  \int_{\mathbb T} & \left( \int_{\Gamma_{\frac 12}(\xi)} |f(z)|^{p} \, |g'(z)|^p \,(1-|z|)^p\,\frac{dm(z)}{1-|z|}\right)^{q/p}  |d\xi|\\
  & \lesssim  \int_{\T}  \left(\int_{\Gamma_{\frac 12}(\xi)\cap G_c} |f(z)|^{p}\ |g'(z)|^p \, (1-|z|)^p \,\frac{dm(z)}{1-|z|}\right)^{q/p}\ |d\xi|\\
  &\quad + \int_{\T}  \left(\int_{\Gamma_{\frac 12}(\xi)\setminus G_c} |f(z)|^{p}\ |g'(z)|^p \, (1-|z|)^p \,\frac{dm(z)}{1-|z|}\right)^{q/p}\ |d\xi|\\
  & \leq \|g\|^q_{\mathcal B}\, \int_{\T}  \left(\int_{\Gamma_{\frac 12}(\xi)\cap G_c} |f(z)|^{p} \,\frac{dm(z)}{1-|z|}\right)^{q/p}\ |d\xi|\\
  &\quad + c^{q} \, \int_{\T}  \left(\int_{\Gamma_{\frac 12}(\xi)} |f(z)|^{p}\ \,\frac{dm(z)}{1-|z|}\right)^{q/p}\ |d\xi|\,.
\end{align*}
Taking into account all the above, we conclude that 
\begin{align*}
	\|g\|^q_{\mathcal B}\,\int_{\T} & \left(\int_{\Gamma_{\frac 12}(\xi)\cap G_c} |f(z)|^{p} \,\frac{dm(z)}{1-|z|}\right)^{q/p}\ |d\xi| + c^{q} \, \int_{\T}  \left(\int_{\Gamma_{\frac 12}(\xi)} |f(z)|^{p}\ \,\frac{dm(z)}{1-|z|}\right)^{q/p}\ |d\xi|\\
	&\quad  \,\,\,\geq  C'' \,\int_{\T}  \left(\int_{\Gamma_{\frac 12}(\xi)} |f(z)|^{p}\ \frac{dm(z)}{1-|z|}\right)^{q/p} |d\xi|\,.
\end{align*}
 % Applying above the 
%\begin{align*}
 % \int_{\T} & \left(\int_{\Gamma_{\frac 12}(\xi)\setminus G_c} |f(z)|^{p}\ |g'(z)|^p \, (1-|z|)^p \,\frac{dm(z)}{1-|z|}\right)^{q/p}\ |d\xi|\\
 % & \leq  c^{q} \, \int_{\T}  \left(\int_{\Gamma_{\frac 12}(\xi)} |f(z)|^{p}\ \,\frac{dm(z)}{1-|z|}\right)^{q/p}\ |d\xi|\,,
%\end{align*}
Therefore, choosing $c>0$ small enough,  we  end up to the strength of  condition $(a)$ of Theorem~\ref{MainTheorem}. As a consequence, there exist an $\eta \in (0,1)$ and a $\delta >0$  such that 
\begin{align*}
	m(G_c\cap \Delta(\alpha,\eta))\geq \delta m(\Delta(\alpha,\eta))\,
\end{align*}
for every $\alpha\in\D$.
\end{proof}

The argument for $M_g$, where $g\in H^{\infty}$, is the same as that for $T_g$. We only have to use the level set $G_c=\{z \in \mathbb D \,:\, |g(z)| >c \}$.  
The approach for $S_g$ makes use of auxiliary lemmas based on 
the  quantity $|f'(z)| (1-|z|)$
instead of $|f(z)|\,.$
 Nevertheless, these new versions of auxiliary lemmas are proved 
 through minor modifications of those we have presented for $T_g$. Consequently, we only present the statements and we leave the details for the interested reader. 
\begin{lemma}\label{L1.2}
Let $1\leq p<\infty$, $\eta \in (0,1)$, $\varepsilon>0$, $f\in \mathcal{H}(\D)$ and
	\begin{align*}
		\mathcal{A}=\left\{\alpha\in\D\ : (1-|\alpha|)^{p}|f'(\alpha)|^{p}<\frac{\varepsilon}{m(\Delta_\eta(\alpha))} \int_{\Delta_\eta(\alpha)} (1-|z|)^{p}|f'(z)|^{p}\ dm(z)\right\}.
	\end{align*}
	Then, there is a constant $C_1=C_1(\eta)>0$ such that
	\begin{align*}
		\int_{\mathcal{A}\cap \Gamma_{1/2}(\xi)} (1-|z|)^p|f'(z)|^{p}\ \frac{dm(z)}{1-|z|}\leq \varepsilon C_1 \int_{ \Gamma_{\beta}(\xi)} (1-|z|)^p|f'(z)|^{p}\ \frac{dm(z)}{1-|z|}
	\end{align*}
for a non tangential region $\Gamma_{\beta}(\xi)$ with vertex at the same point $\xi \in \mathbb T$ as $\Gamma_{1/2}(\xi)$\,.
\end{lemma}

Let $1\leq p<\infty$ and $\lambda\in (0,1)$. We define the set
\begin{align*}
	E_\lambda(\alpha)= \left\{z\in\Delta_\eta(\alpha)\ :\ (1-|z|)^p|f'(z)|^p\geq \lambda (1-|\alpha|)^p|f'(\alpha)|^p\right\}
\end{align*} 
and 
\begin{align*}
	B_\lambda f(\alpha)= \frac{1}{m(E_\lambda(\alpha))}\int_{E_\lambda(\alpha)} |f'(z)|^p (1-|z|)^{p}\ dm(z),\quad \alpha\in\D.
\end{align*}

\begin{lemma}\label{L2.2}
	If $1\leq p<\infty$,  $\varepsilon\in (0,1)$, $f\in \mathcal{H}(\D)$, $\lambda\in (0,\frac{1}{2^p})$ and
	\begin{align*}
		{B}= \left\{\alpha\in\D\ : |f'(\alpha)|(1-|\alpha|)^{p}<\varepsilon^{1+\frac{2}{p}}B_{\lambda}(f)(\alpha)\right\}.
	\end{align*}
	Then, there is a constant $C_2=C_2(\eta)>0$ such that
	\begin{align*}
		\int_{{B}\cap \Gamma_{1/2}(\xi)} (1-|z|)^{p}|f'(z)|^{p}\ \frac{dm(z)}{1-|z|}\leq \varepsilon C_2 \int_{\Gamma_{\beta}(\xi)} (1-|z|)^{p}|f'(z)|^{p}\ \frac{dm(z)}{1-|z|}
	\end{align*}
for a non tangential region $\Gamma_{\beta}(\xi)$ with vertex at the same $\xi \in \mathbb T$ as $\Gamma_{1/2}(\xi)$\,.
\end{lemma}
So that, the analogue of Theorem~\ref{MainTheorem} is the following result.
\begin{theorem}\label{2MainTheorem}
	Let $1\leq p,q<+\infty$ and $G\subset \D$ be a measurable subset. The following assertions are equivalent:
	\begin{enumerate}
		\item[(a)] There is a $\beta \in (0,1)$ and constant $K>0$ such that
		\begin{align*}
			\left(\int_{\T} \left(\int_{\Gamma(\xi)\cap G} |f'(z)|^{p} (1-|z|)^{p}\ \frac{dm(z)}{1-|z|}\right)^{q/p}\ |d\xi|\right)^{1/q}\geq K \|f\|_{T_{p}^{q}},
		\end{align*}
		for all $f\in AT_{p}^{q}$.
		\item[(b)] There exist an $\eta\in(0,1)$ and a $\delta>0$ such that
		\begin{align*}
			m(G\cap \Delta(\alpha,\eta))\geq \delta m(\Delta(\alpha,\eta))
		\end{align*}
		for $\alpha \in \D$.
		\item[(c)] There exist an $\eta\in(0,1)$ and a $\delta>0$ such that
		\begin{align*}
			m(G\cap \Delta_\eta(\alpha))\geq \delta m(\Delta_\eta(\alpha))
		\end{align*}
		for $\alpha \in \D$.
	\end{enumerate}
\end{theorem}
Putting everything together, we get the desired result for the operator $S_g$.
\begin{theorem} 
	Let $1\leq p,q<\infty$ and $g\in H^\infty$. The operator $S_g: RM(p,q)/\mathbb C \rightarrow RM(p,q)/\mathbb C$ has closed range if and only if there is an $\eta \in (0,1)$, a $\delta >0$ and $c>0$ such that
	\begin{align*}
		m(G_c\cap \Delta(\alpha,\eta))\geq \delta m(\Delta(\alpha,\eta)),
	\end{align*}
	for every $\alpha\in\D$, where $G_c=\{z\in\D\ :\ |g(z)|>c\}$.
\end{theorem}

%%%%%%%%%%%%%%%%%%%%%%%%%%%%%%%%%%%%%%%%%%%%%%%%%%%%%%%%%%%%%%%%%%
%%%%%%%%%%%%%%%%%%%%%%%%%%%%%%%%%%%%%%%%%%%%%%%%%%%%%
%%%%%%%%%%%%%%%%%%%%%%%%%%%%%%%%%%%%%%%%%%%%%%%%%%%%%
\section{Examples}
\begin{example}
The univalent functions does not satisfy the condition~\eqref{Tg condition}. Assume that there is a univalent function $g$ satisfying this condition. Using \cite[Theorem 3.4.9.]{Bracci_Contreras_Diaz-Madrigal} we have that
$$|g'(z)|(1-|z|)\leq d(g(z),\partial g(\D))\leq \frac{1}{4}|g'(z)|(1-|z|),$$
for a univalent function $g$ and $z\in\D$. So, this implies that $d(g(\D),\partial g(\D))\geq c$. Moreover, it is known that there is $\xi\in\T$ such that $\angle\lim\limits_{z\rightarrow\xi} g(z)=L\in\C$. Then, by condition~\eqref{Tg condition} we can find a sequence $\{z_k\}$ in the non-tangential region such that $z_k\rightarrow \xi$. However, we have
$$|g(z_k)-L|\geq d(g(z_k),\partial g(\D))\geq c$$
for all $k\in \N$. Therefore, we obtain a contradiction.
\end{example}
%%%%%%%%%%%%%%%%%%%%%%%%%%%%%%%%%%%%%%%%%%%%%%%%%%%%%%%%%%%%%%%%%%
\begin{proposition}
There are functions $f\in\mathcal{B}$ satisfying that there are constants $\delta>0$, $\eta>0$ and $c>0$ such that
\begin{align*}
	m(G_c\cap \Delta(\alpha,\eta))\geq \delta m(\Delta(\alpha,\eta)),
\end{align*}
where $G_c=\{z\in\D\ :\ |f'(z)|(1-|z|)>c\}$, for every $\alpha\in\D$.
\end{proposition}

\begin{proof}
Without loss of generality we can assume that $\alpha\in [0,1)$. Following the proof of \cite[Proposition 5.4]{ramey_bounded_1991}, we can show that for a given large positive integer $q$ we can find a function $f\in\mathcal{B}$ and a constant $c>0$ such that
$$ |f'(z)|(1-|z|)\geq c,\quad 1-\frac{1}{q^{k}}\leq |z| \leq 1-\frac{1}{q^{k+\frac{1}{2}}}, \quad k=1,2,\dots. $$ 

First of all, we will show that if $\eta \geq 1-\frac{1}{q}$ we have that
\begin{align*}
\frac{(1-\eta^2)\alpha}{1-\alpha^2\eta^2}+\frac{(1-\alpha^2)\eta}{1-\alpha^2\eta^2}&=\frac{\alpha+\eta}{1+\alpha\eta}\geq 1-\frac{1}{q^{k+\frac{1}{2}}}\\
\frac{(1-\eta^2)\alpha}{1-\alpha^2\eta^2}-\frac{(1-\alpha^2)\eta}{1-\alpha^2\eta^2}&=\frac{\alpha-\eta}{1-\alpha\eta}\leq 1-\frac{1}{q^{k}}
\end{align*}
 for $\alpha\in\left[1-\frac{1}{q^{k-\frac{1}{2}}}, 1-\frac{1}{q^{k+\frac{1}{2}}}\right]$, $k=1,2,\dots$.
 Since the functions $h_\eta(\alpha)=\frac{\alpha+\eta}{1+\alpha\eta}$ and $g_\eta(\alpha)=\frac{\alpha-\eta}{1-\alpha\eta}$ are increasing functions, the previous inequalities have to satisfy the following
 \begin{align*}
 \frac{1-\frac{1}{q^{k-\frac{1}{2}}}+\eta}{1+\eta\left(1-\frac{1}{q^{k-\frac{1}{2}}}\right)}\geq 1-\frac{1}{q^{k+\frac{1}{2}}},\quad 
 \frac{1-\frac{1}{q^{k+\frac{1}{2}}}-\eta}{1-\eta\left(1-\frac{1}{q^{k+\frac{1}{2}}}\right)}\leq 1-\frac{1}{q^{k}}
 \end{align*}
for $k=1,2,\dots$.
Hence, we obtain that $\eta\geq \max\{1-\frac{1}{q},1-\frac{1}{q^{1/2}}\}=1-\frac{1}{q}$. Analogously, one can prove that if $\eta\geq 1-\frac{1}{q^{3/2}}$ we obtain
\begin{align*}
	\frac{(1-\eta^2)\alpha}{1-\alpha^2\eta^2}+\frac{(1-\alpha^2)\eta}{1-\alpha^2\eta^2}&=\frac{\alpha+\eta}{1+\alpha\eta}\geq 1-\frac{1}{q^{\frac{3}{2}}}\\
	\frac{(1-\eta^2)\alpha}{1-\alpha^2\eta^2}-\frac{(1-\alpha^2)\eta}{1-\alpha^2\eta^2}&=\frac{\alpha-\eta}{1-\alpha\eta}\leq 1-\frac{1}{q}
\end{align*}
for $\alpha\in [0,1-\frac{1}{q^{1/2}}]$.

Bearing this in mind, it follows that if $1-\frac{1}{q^{k-\frac{1}{2}}}\leq \alpha\leq 1-\frac{1}{q^{k+\frac{1}{2}}}$, $k=1,2,\dots$, we have
\begin{align*}
m(G_c\cap \Delta(\alpha,\eta))&\geq m\left(\left\{z\in\Delta(\alpha,\eta): 1-\frac{1}{q^k}\leq |z|\leq 1-\frac{1}{q^{k+\frac{1}{2}}}, k=1,2,\dots\right\}\right)\\
&\geq m\left(D\left(1-\frac{q^{1/2}+1}{2q^{k+\frac{1}{2}}},\frac{q^{1/2}-1}{2q^{k+\frac{1}{2}}}\right)\right)= \frac{(q^{1/2}-1)^{2}}{4q^{2k+1}}\\
&\geq \frac{(q^{1/2}-1)^{2}}{4q^{2k+1}} q^{2k-1}(1-\alpha)^{2}= \frac{(q^{1/2}-1)^{2}}{4q^2} (1-\alpha)^2
\end{align*}
for $\eta \geq 1-\frac{1}{q}$. The remaining case ($0\leq \alpha\leq 1-\frac{1}{q^{1/2}}$) follows in the same manner, that is,  
	\begin{align*}
	m(G_c\cap \Delta(\alpha,\eta))&\geq m\left(\left\{z\in\Delta(\alpha,\eta): 1-\frac{1}{q^k}\leq |z|\leq 1-\frac{1}{q^{k+\frac{1}{2}}}, k=1,2,\dots\right\}\right)\\
	&\geq m\left(D\left(1-\frac{q^{1/2}+1}{2q^{\frac{3}{2}}},\frac{q^{1/2}-1}{2q^{\frac{3}{2}}}\right)\right)= \frac{(q^{1/2}-1)^{2}}{4q^{3}}\\
	&\geq  \frac{(q^{1/2}-1)^{2}}{4q^{3}} (1-\alpha)^{2}
\end{align*}
for $\eta\geq 1-\frac{1}{\eta^{3/2}}$.
\end{proof}

\noindent {\bf Acknowledgments.} We are grateful to Professor Manuel D. Contreras and Professor Luis Rodríguez-Piazza for their useful comments.

%%%%%%%%%%%%%%%%%%%%%%%%%%%%%%%%%%%%%%%%%%%%%%%%%%%%%%%%%%%%%%%%%%

\end{document}